\documentclass[a4paper,USenglish,cleveref, numberwithinsect]{lipics-v2019}
\def\subjclassHeading{}
\ccsdesc{\vskip-\baselineskip}

\usepackage[utf8]{inputenc}

\usepackage{csquotes}
\usepackage[all,cmtip,2cell,all]{xy}\UseAllTwocells

\usepackage{mathtools}

\theoremstyle{definition}
\newtheorem{defi}[theorem]{Definition}
\newtheorem{rem}[theorem]{Remark}

\theoremstyle{plain}
\newtheorem{conjecture}[theorem]{Conjecture}
\newtheorem{problem}[theorem]{Problem}

\theoremstyle{definition}
\newtheorem{constrInternal}[theorem]{Construction}

\newenvironment{construction}[2][]
{\pushQED{\qed}\begin{constrInternal}[{for Problem~\ref{#2}\ifthenelse{\isempty{#1}}{}{; #1}}]}
	{\popQED\end{constrInternal}}

\usepackage{xspace}
\usepackage{xifthen}

\definecolor{LinkColor}{rgb}{0.55,0.0,0.3}
\definecolor{CiteColor}{rgb}{0.55,0.0,0.3}
\definecolor{UrlColor}{rgb}{0.1,0.0,0.88}
\hypersetup{%
  linktocpage=true, pdfstartview=FitV,
  breaklinks=true, pageanchor=true, pdfpagemode=UseOutlines,
  plainpages=false, bookmarksnumbered, bookmarksopen=true, bookmarksopenlevel=3,
  hypertexnames=true, pdfhighlight=/O,
  colorlinks=true,linkcolor=LinkColor,citecolor=CiteColor,
  urlcolor=UrlColor
}

\crefname{defi}{Definition}{Definitions}

\usepackage{microtype}

\newcounter{nameOfYourChoice}
\newcommand{\itemref}[1]{{{\ref{#1}.}}}

\newcommand{\UniMath}{\href{https://github.com/UniMath/UniMath}{\nolinkurl{UniMath}}\xspace}

\newcommand{\longhash}{c26d11bef0b4d4c226ca854d8de340d684c2a10a}
\newcommand{\shorthash}{c26d11b}

\newcommand{\coqdocbasebaseurl}{https://unimath.github.io/doc/UniMath/\shorthash /}

\newcommand{\coqdocbaseurl}{\coqdocbasebaseurl /UniMath.}
\newcommand{\urlhash}{\#}
\newcommand{\coqdocurl}[2]{\coqdocbaseurl #1.html\urlhash #2}
\newcommand{\nolinkcoqident}[1]{\nolinkurl{#1}} 
\makeatletter
\newcommand{\coqident}{\begingroup\@makeother\#\@coqident}
\newcommand{\@coqident}[3][]{%
  \ifthenelse{\isempty{#2}}%
  {\nolinkcoqident{#3}}%
  {\ifthenelse{\isempty{#1}}%
  {\href{\coqdocurl{#2}{#3}}{\nolinkcoqident{#3}}}%
  {\href{\coqdocurl{#2}{#3}}{\nolinkcoqident{#1}}}}%
\endgroup}
\newcommand{\coqfile}[2]{%
  \ifthenelse{\isempty{#1}}%
  {\href{\coqdocbaseurl #2.html}{\nolinkcoqident{#2.v}}}%
  {\href{\coqdocbaseurl #1.#2.html}{\nolinkcoqident{#2.v}}}}
\makeatother

\newcommand{\fat}[1]{\textbf{#1}}
\usepackage[shortcuts]{extdash}
\newcommand{\onecategory}{(1\=/)category\xspace}
\newcommand{\onecategories}{(1\=/)categories\xspace}

\newcommand{\cf}{\emph{c.f.} }
\newcommand{\eg}{\emph{e.g.,} }

\newcommand{\eqdef}{:\equiv}

\newcommand{\pred}{P}

\newcommand{\tc}{\theta}
\newcommand{\tcB}{\gamma}
\newcommand{\tcC}{\tau}
\newcommand{\lunitor}{\lambda}
\newcommand{\runitor}{\rho}
\newcommand{\assoc}{\alpha}
\newcommand{\modvar}{\Gamma}
\newcommand{\identitor}[1]{{#1}_i}
\newcommand{\compositor}[1]{{#1}_c}

\newcommand{\constfont}[1]{\ensuremath{\mathsf{#1}}}
\newcommand{\cat}[1]{\ensuremath{\constfont{#1}}\xspace}
\newcommand{\C}{\cat{C}}
\newcommand{\D}{\cat{D}}
\newcommand{\E}{\cat{E}}
\newcommand{\B}{\cat{B}}
\newcommand{\Cat}{\cat{Cat}}
\newcommand{\Onetypes}{\mbox{1-}\cat{Type}}
\newcommand{\Grpd}{\cat{Grpd}}
\newcommand{\Set}{\cat{Set}}
\newcommand{\homC}[3]{\underline{#1(#2,#3)}}

\newcommand{\plunitor}{\constfont{p}_{\lambda}}
\newcommand{\prunitor}{\constfont{p}_{\rho}}
\newcommand{\passoc}{\constfont{p}_{\alpha}}

\newcommand{\ap}[2]{\constfont{ap} \ #1 \ #2}

\newcommand{\function}{\rightarrow}
\newcommand{\mor}[2]{#1 \onecell #2}

\newcommand{\iso}{\cong}
\newcommand{\adjequiv}{\simeq}
\newcommand{\invtwocell}{\constfont{inv2cell}}
\newcommand{\id}{\operatorname{id}}
\newcommand{\vcomp}{\bullet}
\newcommand{\whiskerl}{\vartriangleleft}
\newcommand{\whiskerr}{\vartriangleright}
\newcommand{\onecell}{\rightarrow}
\newcommand{\mytwocell}{\Rightarrow}
\newcommand{\idtotwomor}{\constfont{idto2cell}}
\newcommand{\idleft}{\constfont{idleft}}
\newcommand{\idright}{\constfont{idright}}
\newcommand{\assoceq}{\constfont{assoc}}

\newcommand{\disp}[1]{\bar{#1}}
\newcommand{\widedisp}[1]{\overline{#1}}
\renewcommand{\aa}{\disp{a}}
\newcommand{\aaa}{\widedisp{a'}}
\newcommand{\bb}{\disp{b}}
\newcommand{\cc}{\disp{c}}
\newcommand{\dd}{\disp{d}}
\newcommand{\ff}{\disp{f}}
\newcommand{\fff}{\widedisp{f'}}
\renewcommand{\gg}{\disp{g}}
\newcommand{\hh}{\disp{h}}
\newcommand{\rr}{\disp{r}}
\newcommand{\xx}{\disp{x}}
\newcommand{\yy}{\disp{y}}
\newcommand{\dtc}{\disp{\tc}}
\newcommand{\dtcB}{\disp{\tcB}}

\newcommand{\FF}{\disp{F}}
\newcommand{\FFo}{\disp{F'}}
\newcommand{\FFt}{\disp{F''}}
\newcommand{\GG}{\disp{G}}
\newcommand{\LL}{\disp{L}}
\newcommand{\RR}{\disp{R}}
\newcommand{\etaeta}{\disp{\eta}}
\newcommand{\thetatheta}{\disp{\theta}}
\newcommand{\epseps}{\disp{\epsilon}}
\newcommand{\mm}{\disp{m}}

\newcommand{\dob}[2]{\ensuremath{{#1}_{#2}}} 
\newcommand{\dmor}[3]{#1 \xrightarrow{#3} #2} 
\newcommand{\dtwo}[3]{#1 \xRightarrow{#3} #2} 
\newcommand{\total}[2][]{\ensuremath{\textstyle \int_{#1}{#2}}} 
\newcommand{\dsigma}[2][]{\ensuremath{\sum_{#1}{#2}}} 
\newcommand{\diso}[3]{\ensuremath{ {#1} \iso_{#2} {#3} }}
\newcommand{\dadjequiv}[3]{\ensuremath{ {#1} \adjequiv_{#2} {#3} }}
\newcommand{\did}{\id}
\newcommand{\dproj}{\pi}

\newcommand{\Base}{\cat{Base}}
\newcommand{\mapod}{\cat{Map1D}}
\newcommand{\mapo}{\cat{Map1}}
\newcommand{\maptd}{\cat{Map2D}}
\newcommand{\mapid}{\cat{MapId}}
\newcommand{\mapcd}{\cat{MapC}}
\newcommand{\rpseudo}{\cat{RawPseudo}}
\newcommand{\pseudo}{\cat{Pseudo}}
\newcommand{\pseudoF}[2]{#1 \onecell #2}
\newcommand{\pstrans}[2]{#1 \mytwocell #2}
\newcommand{\modif}[2]{#1 \Rrightarrow #2}
\newcommand{\disppsfun}[3]{\dmor{#1}{#2}{#3}}
\newcommand{\disppstrans}[3]{\dtwo{#1}{#2}{#3}}
\newcommand{\dispmodif}[3]{\xymatrix{#1 \ar@3[r]^-{#3} & #2}}

\newcommand{\algd}{\cat{Alg_D}}
\newcommand{\alg}{\cat{Alg}}
\newcommand{\addcell}{\cat{Add2Cell}}

\newcommand{\PM}{\cat{M_1}}
\newcommand{\LM}{\cat{M_2}}
\newcommand{\M}{\cat{M}}

\newcommand{\K}{\cat{K}}

\newcommand{\PShD}{\cat{PShD}}
\newcommand{\TT}{\cat{CwF_1}}
\newcommand{\Ty}{\cat{Ty}}
\newcommand{\Tm}{\cat{Tm}}
\newcommand{\p}{\cat{p}}
\newcommand{\CwFd}{\cat{dCwF_2}}
\newcommand{\CwFh}{\cat{CwF_2}}
\newcommand{\isCwF}{\cat{isCwF}}
\newcommand{\CwF}{\cat{CwF}}

\newcommand{\image}{\constfont{im}}
\newcommand{\restrict}[1]{\overline{#1}}
\newcommand{\completion}[1]{\mathcal{RC}(#1)}
\newcommand{\loccompletion}[1]{\mathcal{RC}_{\constfont{loc}}(#1)}
\newcommand{\etaloc}{\ensuremath{\eta_{\constfont{loc}}}}
\newcommand{\etaglob}{\ensuremath{\eta_{\constfont{glob}}}}

\newcommand{\idtoiso}{\constfont{idtoiso}}
\newcommand{\dispidtoiso}{\constfont{disp\_idtoiso}}
\newcommand{\Iso}[3][]{\constfont{Iso}_{#1}(#2,#3)} 
\newcommand{\CatIso}[3][]{\constfont{CatIso}_{#1}(#2,#3)} 
\newcommand{\AdjEquiv}[3][]{\constfont{AdjEquiv}_{#1}(#2,#3)} 

\newcommand{\spac}{\hskip 0.2em plus 0.1em}
\def\Lam #1.{\lambda\,#1.\spac}%
\def\Sum #1.{\sum_{#1}\spac}%
\def\Prod #1.{\prod_{#1}\spac}%

\newcommand{\repps}{\constfont{Rep}_0}
\newcommand{\reptr}{\constfont{Rep}_1}
\newcommand{\repmo}{\constfont{Rep}_2}
\newcommand{\yoneda}{\constfont{y}}

\newcommand{\type}[1]{\operatorname{\textsf{#1}}}
\newcommand{\constructor}[1]{\operatorname{\mathsf{#1}}}

\newcommand{\depeq}[1][*]{=_{#1}}                                                                                      
\newcommand{\defeq}{:\equiv}

\newcommand{\hProp}[0]{\type{hProp}}

\newcommand{\U}[0]{\type{U}}

\newcommand{\idpath}{\constructor{refl}}
\newcommand{\J}{\constfont{J}}
  
\newcommand{\op}[1]{#1^{\constructor{op}}}

\newcommand{\pathgroupoid}{\constfont{PathGrpd}}
\newcommand{\obs}{\constfont{Ob}}

\newcommand{\fbg}[1]{\pi(#1)}
\newcommand{\apPS}[1]{\overline{#1}}
\newcommand{\apPT}[1]{\overline{#1}}
\newcommand{\apMOD}[1]{\overline{#1}}

\newcommand{\twoar}{\ar@{=>}}

\usepackage[draft]{optional}
\newcommand{\plan}[1]{}
\newcommand{\BA}[1]{}
\newcommand{\NW}[1]{}
\newcommand{\DF}[1]{}
\opt{draft}{
\renewcommand{\plan}[1]{\textcolor{blue}{#1}\PackageWarning{TODO}{TODO: #1}}
\renewcommand{\BA}[1]{\textcolor{orange}{BA: #1}\PackageWarning{TODO}{TODO: #1}}
\renewcommand{\NW}[1]{\textcolor{purple}{NW: #1}\PackageWarning{TODO}{TODO: #1}}
\renewcommand{\DF}[1]{\textcolor{magenta}{DF: #1}\PackageWarning{TODO}{TODO: #1}}
}

\title{Bicategories in Univalent Foundations}

\titlerunning{Bicategories in Univalent Foundations}

\author{Benedikt Ahrens}
       {Delft University of Technology, The Netherlands \and University of Birmingham, United Kingdom}
       {B.P.Ahrens@tudelft.nl}{https://orcid.org/0000-0002-6786-4538}
       {This material is based upon work supported by the Air Force Office of Scientific Research under award number FA9550-17-1-0363.
        Ahrens acknowledges the support of the Centre for Advanced Study (CAS) in Oslo, Norway, which funded and hosted the research project \emph{Homotopy Type Theory and Univalent Foundations} during the 2018/19 academic year.
       }

\author{Dan Frumin}
       {University of Groningen, The Netherlands}
       {d.frumin@rug.nl}
       {https://orcid.org/0000-0001-5864-7278}
       {Supported by the Netherlands Organisation for Scientific Research (NWO/TTW) under the STW project 14319 and VIDI Project No.  016.Vidi.189.046.}

\author{Marco Maggesi}
       {Dipartimento di Matematica e Informatica ``Dini'', Università degli Studi di Firenze, Italy}
       {marco.maggesi@unifi.it}{https://orcid.org/0000-0003-4380-7691}
       {MIUR, GNSAGA-INdAM.}
\author{Niccolò Veltri}
       {Department of Software Science, Tallinn University of Technology, Estonia}
       {niccolo@cs.ioc.ee}
       {https://orcid.org/0000-0002-7230-3436}
       {This research was supported by the Estonian Research Council grant PSG659 and by the ESF funded Estonian IT Academy
  research measure (project 2014-2020.4.05.19-0001).}
\author{Niels van der Weide}
       {Institute of Computation and Information Science, Radboud University, The Netherlands \and School of Computer Science, University of Birmingham, United Kingdom}
       {nnmvdw@gmail.com}
       {https://orcid.org/0000-0003-1146-4161}
       {}

\authorrunning{B. Ahrens, D. Frumin, M. Maggesi, N. Veltri, and N. van der Weide}

\Copyright{Benedikt Ahrens, Dan Frumin, Marco Maggesi, Niels van der Weide}


\category{}

\relatedversion{}

\supplement{}

\acknowledgements{
    We thank Peter LeFanu Lumsdaine and Michael Shulman for helpful discussions on the subject matter, and James Leslie for pointing out several typos in the conference version of this article.
 	We furthermore thank the referees for their careful reading and thoughtful and constructive criticism.
Work on this article was supported by a grant from the \href{https://eutypes.cs.ru.nl/}{COST Action EUTypes CA15123}.
 	We would like to express our gratitude to all the EUTypes actors for their support.
 	This work was partially funded by EPSRC under agreement number EP/T000252/1.
 	}

\nolinenumbers 

\hideLIPIcs  

\begin{document}

\maketitle

\begin{abstract}
We develop bicategory theory in univalent foundations.
Guided by the notion of univalence for \onecategories studied by Ahrens, Kapulkin, and Shulman, 
we define and study univalent bicategories.
To construct examples of univalent bicategories in a modular fashion, we develop \emph{displayed bicategories}, an analog of displayed 1-categories introduced by Ahrens and Lumsdaine.
We demonstrate the applicability of this notion, and prove that several bicategories of interest are univalent.
Among these are the bicategory of univalent categories with families and the bicategory of pseudofunctors between univalent bicategories.
Furthermore, we show that every bicategory with univalent hom-categories is weakly equivalent to a univalent bicategory.

All of our work is formalized in Coq as part of the \UniMath library of univalent mathematics.
\end{abstract}

\tableofcontents

\section{Introduction}

Category theory (by which we mean 1-category theory) 
is established as a convenient language to structure and discuss mathematical objects and morphisms between them.
To axiomatize the fundamental objects of category theory itself---categories, functors, and natural transformations---the theory of 1-categories
is not enough.
Instead, category-like structures allowing for ``morphisms between morphisms'' were developed to account for the natural transformations.
Among those structures are \emph{bicategories}.
Bicategory theory was originally developed by Bénabou \cite{10.1007/BFb0074299} in set-theoretic foundations.
The goal of our work is to develop bicategory theory in univalent foundations.
Specifically, we give a notion of a \emph{univalent bicategory} and show that some bicategories of interest are univalent, with examples from algebra and type theory. 
To this end, we generalize \emph{(univalent) displayed categories} of Ahrens and Lumsdaine \cite{AhrensL19} to the bicategorical setting,
and prove that the total bicategory generated by a displayed bicategory is univalent, if the
constituent pieces are.
In addition, we show how to embed any bicategory with univalent hom-categories into a univalent bicategory via the Yoneda lemma,
and we show how to use displayed machinery to construct biequivalences between total bicategories.

\subparagraph*{Univalent foundations and categories therein}
According to Voevodsky \cite{Voevodsky}, a foundation of mathematics specifies, in particular, three things:
\begin{enumerate}
 \item a language for mathematical objects;
 \item a notion of proposition and proof; and
 \item an interpretation of those into a world of mathematical objects.
\end{enumerate}
By \enquote{univalent foundations}, we mean the foundation given by
univalent type theory as described, \eg in the HoTT book \cite{hottbook},
with its notion of \enquote{univalent logic}, and
the interpretation of univalent type theory in Kan complexes
expected to arise from Voevodsky's simplicial set model \cite{simpset}.

In the simplicial set model, univalent categories (just called \enquote{categories} in \cite{rezk_completion})
correspond to truncated complete Segal spaces, which in turn are 
equivalent to ordinary (set-theoretic) categories.
In this respect,
univalent categories are \enquote{the right} notion of categories in univalent foundations:
they correspond exactly to the traditional set-theoretic notion of category.
Similarly, the notion of \emph{univalent} bicategory,
studied in this paper, provides the correct notion of bicategory in univalent foundations---see, e.g., \cite[Example~9.1]{univalence-principle}.
In this work, we provide results for showing, modularly, that certain bicategories are univalent.

Throughout this article, we work in type theory with function extensionality.
We explicitly mention any use of the univalence axiom.
We use the notation standardized in \cite{hottbook};
a significantly shorter overview of the setting we work in is given in \cite{rezk_completion}.
As a reference for 1-category theory in univalent foundations, 
we refer to \cite{rezk_completion}, which follows a path suggested by Hofmann and Streicher \cite[Section 5.5]{MR1686862}.

\subparagraph*{Motivation: bicategories for type theory}
One of the motivations for this work stems from
several particular (classes of) bicategories that come up in our work on the semantics of type theories
and Higher Inductive Types (HITs).

Firstly, we are interested in the 
\enquote{categories with structure} that have been used in the 
model theory of type theories.
The purpose of the various categorical structures is to model context extension and substitution.
Prominent such notions are categories with families (see, \eg \cite{DBLP:journals/mscs/ClairambaultD14,DBLP:conf/types/Dybjer95}),
categories with attributes (see, \eg \cite{PittsAM:catl}), and categories with display maps (see, \eg \cite{DBLP:books/daglib/0031002,north_2019}).
Each notion of \enquote{categorical structure} gives rise to a bicategory whose objects are categories equipped with 
such a structure.
In the present work, we provide machinery that can be used to show, in a modular way, that these bicategories are univalent; we exemplify the machinery with categories with families.

Secondly, Dybjer and Moeneclaey define a notion of signature for 1-HITs and study algebras of those signatures \cite{DBLP:journals/entcs/DybjerM18}.
These algebras are groupoids equipped with extra structure according to the signature.
In the present work, we give general methods for constructing bicategories of such algebras
and we demonstrate the usage of those methods by constructing the bicategory of monads internal to a given bicategory.
We then construct a bicategory of Kleisli triples (an alternative presentation of monads\footnote{Also known in the literature as extension systems~\cite{marmolejo2010monads} or Manes-style monads~\cite{manes:algtheories}.}), and show that it is equivalent to the bicategory of monads.
We also show that the resulting bicategory of monads internal to the bicategory of univalent categories is biequivalent to the bicategory of Kleisli triples.

\subparagraph*{Technical contribution: displayed bicategories}
In this work, we develop the notion of \emph{displayed bicategory} in analogy to the 1-categorical notion
of displayed category introduced in \cite{AhrensL19}.
Intuitively, a displayed bicategory $\D$ over a bicategory $\B$ represents data and properties to be added to $\B$ to form a new bicategory:
$\D$ gives rise to the \emph{total bicategory} $\total{\D}$.
Its cells are pairs $(b,d)$ where $d$ in $\D$ is a \enquote{displayed cell} over $b$ in $\B$.
Univalence of $\total{\D}$ can be shown from univalence of $\B$ and \enquote{displayed univalence} of $\D$.
The latter two conditions are easier to show, sometimes significantly easier.

Two features make the displayed point of view particularly useful:
firstly, displayed structures can be \emph{iterated}, making it possible to build bicategories of very complicated objects
layerwise.
Secondly, displayed \enquote{building blocks} can be provided, for which univalence is proved once and for all.
These building blocks, \eg cartesian product, can be used like LEGO\textsuperscript{\texttrademark} pieces to \emph{modularly} build 
bicategories of large structures that are automatically accompanied by a proof of univalence.

We demonstrate these features in examples, proving univalence of three important (classes of) bicategories:
first, the bicategory of pseudofunctors between two univalent bicategories;
second, bicategories of algebraic structures (given as pseudoalgebras of pseudofunctors);
and third, the bicategory of categories with families.

\subparagraph*{Main contributions}
Here we give a list of the main results presented in this paper:
\begin{itemize}
\item Following Ahrens, Kapulkin, and Shulman's construction of the
  Rezk completion for categories \cite[Theorem 8.5]{rezk_completion},
  we show in \Cref{sec:yoneda} that every locally univalent bicategory embeds into a
  univalent one. This result fundamentally relies on the proof of a
  bicategorical version of the Yoneda lemma.
\item We develop displayed infrastructure for bicategories and show that it is useful for building
  bicategories. In particular, we modularly prove univalence of complicated
  bicategories in \Cref{sec:examples}, such as the bicategory of
  pseudofunctors between two univalent bicategories, the bicategory of
  pseudoalgebras of a given pseudofunctor, and the bicategory of
  categories with families.
\item We show the benefits of the displayed infrastructure for
  defining morphisms between bicategories in layers. We demonstrate
  this on two examples in \Cref{sec:dispconstr}: the construction of a
  biequivalence between pointed 1-types and pointed univalent
  groupoids and the construction of a biequivalence between monads
  internal to the bicategory of univalent categories and the
  bicategory of Kleisli triples.
\end{itemize}

\subparagraph*{Formalization}
The results presented here are mechanized in the \UniMath library \cite{UniMath},
which is based on the Coq proof assistant \cite{Coq:Manual}.
The \UniMath library is under constant development; in this paper, we refer to the version
with \texttt{git} hash \href{https://github.com/UniMath/UniMath/tree/\longhash}{\shorthash}.
Throughout the paper, definitions and statements are accompanied by a link to the online documentation of that version. 
For instance, the link \coqident{Bicategories.Core.Bicat}{bicat} points to the definition of a bicategory.

\subparagraph*{Related work}
Our work extends the notion of univalence from 1-categories \cite{rezk_completion}
to bicategories. 
Similarly, we extend the notion of displayed 1-category \cite{AhrensL19} to the bicategorical setting.

Ahrens, North, Shulman, and Tsementzis \cite{DBLP:conf/lics/AhrensNST20} devise a notion of ``signature'' and ``theory'' for mathematical structures. To each theory they associate a type of models of that theory, and a predicate of ``being univalent'' on such models.
Their signatures encompass, in particular, bicategories~\cite[Example~9.1]{univalence-principle}, more specifically, saturated ana-bicategories.
Ana-bicategories that are both saturated and univalent should correspond to the univalent bicategories studied here, even though a formal statement and construction of a suitable equivalence is outside the scope of the present work.

Capriotti and Kraus \cite{DBLP:journals/pacmpl/CapriottiK18} study univalent $(n,1)$-categories for $n \in \{0,1,2\}$.
They only consider bicategories where the 2-cells are equalities between 1-cells;
in particular, all 2-cells in \cite{DBLP:journals/pacmpl/CapriottiK18} are invertible, and their $(2,1)$-categories are 
by definition locally univalent (cf.~\Cref{def:univalence}, \Cref{def-item:local-univalence}). 
Consequently, the condition called \emph{univalence} by Capriotti and Kraus is what we call \emph{global univalence}, 
cf.\ \Cref{def:univalence}, \Cref{def-item:global-univalence}.
In this work, we study bicategories, a.k.a.\ (weak) $(2,2)$-categories, that is, we allow for non-invertible 2-cells.
The examples we study in \Cref{sec:examples} are proper $(2,2)$-categories and are not covered by
\cite{DBLP:journals/pacmpl/CapriottiK18}.

Outside of univalent foundations, there are also computer-checked libraries of bicategory theory, see, \eg~\cite{DBLP:journals/afp/Stark20,DBLP:conf/cpp/HuC21}.

\subparagraph*{Publication history}
This article is an extended version of a conference contribution \cite{DBLP:conf/rta/AhrensFMW19}.
Compared to the conference version, we have added the following content:
\begin{itemize}
 \item In \Cref{sec:bicategories}, we define the notion of
   biequivalence of bicategories, the ``correct'' notion of sameness
   for bicategories. We construct a biequivalence between 1-types and
   univalent groupoids.
 \item In \Cref{sec:univalence}, we present an induction principle for
   invertible 2-cells in a locally univalent bicategory and an
   induction principle for adjoint equivalences in a globally
   univalent bicategory. We put these principles to work in a number
   of examples.
 \item \Cref{sec:2-categories} is new. In there, we propose a
   definition of 2-category and of strict bicategory, and we show that these are equivalent. 
 \item \Cref{sec:yoneda} is new. In there, we show that any bicategory
   embeds into a univalent one via the Yoneda embedding. This
   construction is reminiscent of the Rezk completion for categories.
 \item In \Cref{sec:displayed}, we give the definition of the
   displayed bicategory of monads internal to a given bicategory and
   the displayed bicategory of Kleisli triples . The bicategory of
   monads on a bicategory $\B$ is univalent whenever $\B$ is
   univalent, which is proved in \Cref{sec:algebraic-examples}.
 \item \Cref{sec:dispconstr} is new. In there, we introduce the notion
   of displayed biequivalence. Using this notion, we show that the
   biequivalence between 1-types and univalent groupoids extends to a
   biequivalence between their pointed variants. We also construct a
   biequivalence between the bicategory of Kleisli triples and the bicategory of monads internal to the bicategory of univalent
   categories.
 \item \Cref{sec:disp-inserters} is new. Following a suggestion by an anonymous referee, we generalize the constructions in \Cref{sec:algebraic-examples,sec:cwfs} using displayed inserters.
\end{itemize}

\section{Bicategories and Some Examples}
\label{sec:bicategories}
Bicategories were introduced by Bénabou \cite{10.1007/BFb0074299},
encompassing monoidal categories, 2-categories (in particular, the 2-category of categories),
and other examples.
He (and later many other authors) defines bicategories in the style of
``categories weakly enriched in categories''.
That is, the hom-objects $\B_1(a,b)$ of a bicategory $\B$ are taken to be \onecategories, and composition
is given by a functor $\B_1(a,b)\times \B_1(b,c) \to \B_1(a,c)$.
This presentation of bicategories is concise and convenient for communication between mathematicians.

In this article, we use a different, more unfolded definition of bicategories, which is inspired by Bénabou \cite[Section~1.3]{10.1007/BFb0074299}
and \cite[Section `Details']{nlab:bicategory}.
One the one hand, it is more verbose than the definition via weak enrichment.
On the other hand, it is better suited for our purposes, in particular,
it is suitable for defining \emph{displayed bicategories}, cf.\ \Cref{sec:displayed}.

\begin{defi}[\coqident{Bicategories.Core.Bicat}{prebicat}, \coqident{Bicategories.Core.Bicat}{bicat}]
 \label{def:bicat}
A \fat{prebicategory} $\B$ consists of 
\begin{enumerate}
        \item a type $\B_0$ of \fat{objects}; \label{item:0-cell}
        \item a type $\B_1(a, b)$ of \fat{1-cells} for all $a, b : \B_0$; \label{item:1-cell}
        \item a type $\B_2(f, g)$ of \fat{2-cells} for all $a, b : \B_0$ and $f, g : \B_1(a,b)$; \label{item:2-cell}
	\item an \fat{identity 1-cell} $\id_1(a) : \B_1(a,a)$; \label{item:id-1}
	\item a \fat{composition} $\B_1(a,b) \times \B_1(b,c) \rightarrow \B_1(a,c)$, written $f \cdot g$; \label{item:comp-1}
	\item an \fat{identity 2-cell} $\id_2(f) : \B_2(f,f)$; \label{item:id-2}
	\item a \fat{vertical composition} $\tc \vcomp \tcB : \B_2(f,h)$ for all 1-cells $f,g,h : \B_1(a,b)$ and 2-cells $\tc : \B_2(f,g)$ and $\tcB : \B_2(g,h)$; \label{item:v-comp}
	\item a \fat{left whiskering} $f \whiskerl \tc : \B_2(f \cdot g, f \cdot h)$ for all 1-cells $f : \B_1(a,b)$ and $g,h : \B_1(b,c)$ and 2-cells $\tc : \B_2(g,h)$; \label{item:left-whisker}
	\item a \fat{right whiskering} $\tc \whiskerr h : \B_2(f \cdot h, g \cdot h)$ for all 1-cells $f, g : \B_1(a,b)$ and $h : \B_1(b,c)$ and 2-cells $\tc : \B_2(f,g)$; \label{item:right_whisker}
	\item a \fat{left unitor} $\lunitor(f) : \B_2(\id_1(a) \cdot f,f)$ and its inverse $\lunitor(f)^{-1} : \B_2(f,\id_1(a) \cdot f)$; \label{item:lunitor}
	\item a \fat{right unitor} $\runitor(f) : \B_2(f \cdot \id_1(b),f)$ and its inverse $\runitor(f)^{-1} : \B_2(f,f \cdot \id_1(b))$; \label{item:runitor}
	\item a \fat{left associator} $\assoc(f,g,h) : \B_2(f \cdot (g \cdot h), (f \cdot g) \cdot h)$ and a right associator $\assoc(f,g,h)^{-1} : \B_2((f \cdot g) \cdot h, f \cdot (g \cdot h))$ for $f : \B_1(a,b)$, $g : \B_1(b,c)$, and $h : \B_1(c,d)$ \label{item:lassociator}
	\setcounter{nameOfYourChoice}{\value{enumi}}
\end{enumerate}
such that, for all suitable objects, 1-cells, and 2-cells,
 \begin{enumerate}
         \setcounter{enumi}{\value{nameOfYourChoice}}
	 \item $\id_2(f) \vcomp \tc = \tc, \quad \tc \vcomp \id_2(g) = \tc, \quad \tc \vcomp (\tcB \vcomp \tcC) = (\tc \vcomp \tcB) \vcomp \tcC$; \label{item:vcomp-l-r-assoc}
	 \item $f \whiskerl (\id_2 g) = \id_2(f \cdot g), \quad f \whiskerl (\tc \vcomp \tcB) = (f \whiskerl \tc) \vcomp (f \whiskerl \tcB)$; \label{item:lwhisker-id-comp}
	 \item $(\id_2 f) \whiskerr g = \id_2(f \cdot g), \quad (\tc \vcomp \tcB) \whiskerr g = (\tc \whiskerr g) \vcomp (\tcB \whiskerr g)$; \label{item:rwhisker-id-comp}
	 \item $(\id_1(a) \whiskerl \tc) \vcomp \lambda(g) = \lambda(f) \vcomp \tc$; \label{item:id-lwhisker-vcomp}
	 \item $(\tc \whiskerr \id_1(b)) \vcomp \rho(g) = \rho(f) \vcomp \tc$; \label{item:id-rwhisker-vcomp}
	 \item $(f \whiskerl (g \whiskerl \tc)) \vcomp \alpha(f, g, i) = \alpha(f,g,h) \vcomp ((f \cdot g) \whiskerl \tc)$; \label{item:lwhisker-assoc}
	 \item $(f \whiskerl (\tc \whiskerr i)) \vcomp \alpha(f,h,i)  = \alpha(f,g,i) \vcomp ((f \whiskerl \tc) \whiskerr i)$; \label{item:lwhisker-rwhisker-assoc}
	 \item $(\tc \whiskerr (h \cdot i)) \vcomp \alpha(g,h,i) = \alpha(f,h,i) \vcomp ((\tc \whiskerr h) \whiskerr i)$; \label{item:rwhisker-assoc}
	 \item $(\tc \whiskerr h) \vcomp (g \whiskerl \tcB) = (f \whiskerl \tcB) \vcomp (\tc \whiskerr i)$; \label{item:vcomp-whisker}
	 \item $\lambda(f) \vcomp \lambda(f)^{-1} = \id_2(\id_1(a) \cdot f), \quad \lambda(f)^{-1} \vcomp \lambda(f) = \id_2(f)$; \label{item:lambda}
	 \item $\rho(f) \vcomp \rho(f)^{-1} = \id_2(f \cdot \id_1(b)), \quad \rho(f)^{-1} \vcomp \rho(f) = \id_2(f)$; \label{item:rho}
	 \item $\alpha(f,g,h) \vcomp \alpha(f,g,h)^{-1} = \id_2(f \cdot (g \cdot h)), \quad \alpha(f,g,h)^{-1} \vcomp \alpha(f,g,h) = \id_2((f \cdot g) \cdot h)$; \label{item:alpha}
	 \item $\alpha(f, \id_1(b),g) \vcomp (\rho(f) \whiskerr g) = f \whiskerl \lambda(g)$; \label{item:tri}
\[
\xymatrix{
f \cdot (\id_1(b) \cdot g) \twoar[rr]^{\alpha(f, \id_1(b), g)}
                           \twoar[d]_{f \whiskerl \lambda(g)}
             && (f \cdot \id_1(b)) \cdot g \twoar[dll]^{\rho(f) \whiskerr g} \\
f \cdot g &&
}
\]
	 \item $\alpha(f,g,h \cdot i) \vcomp \alpha(f \cdot g, h, i) = (f \whiskerl \alpha(g,h,i)) \vcomp \alpha(f,g \cdot h, i) \vcomp (\alpha(f,g,h) \whiskerr i)$.\label{item:pent}
\[
\xymatrix{
f \cdot (g \cdot (h \cdot i)) \twoar[rr]^{\alpha(f, g, h \cdot i)} \twoar[d]_{f \whiskerl \alpha(g, h, i)}
               && (f \cdot g) \cdot (h \cdot i) \twoar[rr]^{\alpha(f\cdot g, h, i)}
               && ((f \cdot g) \cdot h) \cdot i \\
f \cdot ((g \cdot h) \cdot i) \twoar[rr]_{\alpha(f, g \cdot h, i)}
               && (f \cdot (g \cdot h)) \cdot i  \twoar[urr]_{\ \ \ \alpha(f, g, h) \whiskerr i}
               &&
}
\]
 \end{enumerate}
 A \fat{bicategory} is a prebicategory whose types of 2-cells $\B_2(f,g)$ are \emph{sets} for all $a,b : \B_0$ and $f,g : \B_1(a,b)$.
\end{defi}

We write $a \onecell b$ for $\B_1(a,b)$ and $f \mytwocell g$ for $\B_2(f,g)$.

Mitchell Riley formalized a definition of bicategories as ``categories weakly enriched in categories'' in \UniMath, based on work by Peter LeFanu Lumsdaine.
We do not reproduce this definition here; it is available as \coqident{Bicategories.WkCatEnrichment.prebicategory}{prebicategory}.
That definition is equivalent to our definition, in the following sense:
\begin{proposition}[\coqident{Bicategories.WkCatEnrichment.hcomp_bicat}{weq_bicat_prebicategory}]
  \label{prop:eqdef_weakenrichment}
  The type of bicategories defined in \Cref{def:bicat} is equivalent to
 the type of bicategories in terms of weak enrichment.
\end{proposition}

For this result, one needs to show that each $\B_1(a,b)$ forms a category whose morphisms are 2-cells.
Let us introduce this formally.

\begin{defi}[\coqident{Bicategories.Core.Bicat}{hom}]
Let $\B$ be a bicategory and $a, b : \B_0$ objects of $\B$.
Then we define the \fat{hom-category} $\homC{\B_1}{a}{b}$ to be the category whose objects are 1-cells $f : a \onecell b$ and whose morphisms from $f$ to $g$ are 2-cells $\alpha : f \mytwocell g$ of $\B$.
The identity morphisms are identity 2-cells and the composition is vertical composition of 2-cells.
\end{defi}

Recall that our goal is to study univalence of bicategories, which is a property that relates equivalence and equality.
For this reason, we study the two analogs of the 1-categorical notion of isomorphism.
The corresponding notion for 2-cells is that of \emph{invertible 2-cells}.
\begin{defi}[\coqident{Bicategories.Core.Bicat}{is_invertible_2cell}]
A 2-cell $\tc : f \mytwocell g$ is called \fat{invertible} if we have $\tcB : g \mytwocell f$ such that $\tc \vcomp \tcB = \id_2(f)$ and $\tcB \vcomp \tc = \id_2(g)$.
An \fat{invertible 2-cell} consists of a 2-cell and a proof that it is invertible, and $\invtwocell(f,g)$ is the type of invertible 2-cells from $f$ to $g$.
\end{defi}
Since 2-cells form a set and inverses are unique, being an invertible 2-cell is a proposition.
In addition, $\id_2(f)$ is invertible, and we write $\id_2(f) : \invtwocell(f,f)$ for this invertible 2-cell.

The bicategorical analog of isomorphisms for 1-cells is the notion of \emph{adjoint equivalence}.
\begin{defi}[\coqident{Bicategories.Core.Adjunctions}{adjoint_equivalence}]
\label{def:adjequiv}
An \fat{adjoint equivalence structure} on a 1-cell $f: a \onecell b$ consists of a 1-cell $g : b \onecell a$ and invertible 2-cells $\eta : \id_1(a) \mytwocell f \cdot g$ and $\varepsilon : g \cdot f \mytwocell \id_1(b)$
such that the following two diagrams commute
\[\xymatrix{
(f \cdot g) \cdot f \twoar[rr]^{\alpha(f,g,f)} & & f \cdot (g \cdot f) \twoar[d]^{f \whiskerl \varepsilon} & \\
\id_1(a) \cdot f \twoar[u]^{\eta \whiskerr f} & & f \cdot \id_1(b) \twoar[d]^{\rho(f)} \\
f \twoar[rr]_{\id_2(f)} \twoar[u]^{\lambda(f)^{-1}} & & f
}
\quad
\xymatrix{
g\cdot(f\cdot g) \twoar[rr]^{\alpha(g,f,g)^{-1}} & & (g\cdot f)\cdot g \twoar[d]^{\varepsilon \whiskerr g} & \\
g \cdot \id_1(b) \twoar[u]^{g \whiskerl \eta} & & \id_1(b) \cdot g \twoar[d]^{\lambda(g)} \\
g \twoar[rr]_{\id_2(g)} \twoar[u]^{\rho(g)^{-1}} & & g
}\]

An \fat{adjoint equivalence} consists of a 1-cell $f$ together with an adjoint equivalence structure on $f$.
The type $\AdjEquiv{a}{b}$ consists of all adjoint equivalences from $a$ to $b$.
\end{defi}
We call $\eta$ and $\varepsilon$ the unit and counit of the adjoint equivalence, and we call $g$ the right adjoint.
The prime example of an adjoint equivalence is the identity 1-cell $\id_1(a)$ and we denote it by $\id_1(a) : \AdjEquiv{a}{a}$.
Sometimes, we write $a \adjequiv b$ for $\AdjEquiv{a}{b}$.

Before we start our study of univalence, we present some examples of bicategories and preliminary notions from bicategory theory.

\begin{example}[\coqident{Bicategories.Core.Examples.TwoType}{fundamental_bigroupoid}]
\label{ex:bigroupoid}
Let $X$ be a 2-type.
Then we define the \fat{fundamental bigroupoid} $\fbg{X}$ to be the bicategory whose 0-cells are inhabitants of $X$, 1-cells from $x$ to $y$ are paths $x = y$, and 2-cells from $p$ to $q$ are higher-order paths $p = q$.
The operations, such as composition and whiskering, are defined using path induction.
Every 1-cell is an adjoint equivalence and every 2-cell is invertible.
\end{example}

\begin{example}[\coqident{Bicategories.Core.Examples.OneTypes}{one_types}]
\label{ex:one_types}
Let $\U$ be a universe.
The objects of the bicategory $\Onetypes_{\U}$ are 1-truncated types of the universe $\U$, the 1-cells are functions between the underlying types, and the 2-cells are homotopies between functions.
The 1-cells $\id_1(X)$ and $f \cdot g$ are defined as the identity and composition of functions, respectively.
The 2-cell $\id_2(f)$ is $\idpath$, the 2-cell $p \vcomp q$ is the concatenation of paths.
The unitors and associators are defined as identity paths.
Every 2-cell is invertible, and adjoint equivalences from $X$ to $Y$ are the same as equivalences of types from $X$ to $Y$.
\end{example}

\begin{example}[\coqident{Bicategories.Core.Examples.BicatOfUnivCats}{bicat_of_univ_cats}]
\label{ex:cat}
We define the bicategory $\Cat$ of univalent categories as the bicategory whose 0-cells are univalent categories, 1-cells are functors, and 2-cells are natural transformations.
The identity 1-cells are identity functors, the composition and whiskering operations are composition of functors and whiskering of functors and transformations, respectively.
Invertible 2-cells are natural isomorphisms, and adjoint equivalences are external adjoint equivalences of categories.
\end{example}

\begin{example}[\coqident{Bicategories.Core.Examples.OpMorBicat}{op1_bicat}]
\label{ex:op}
Let $\B$ be a bicategory.
Then we define $\op{\B}$ to be the bicategory whose objects are objects in $\B$, 1-cells from $x$ to $y$ are 1-cells $y \onecell x$ in $\B$, and the 2-cells from $f$ to $g$ are 2-cells $f \mytwocell g$ in $\B$.
\end{example}

\begin{defi}[\coqident{Bicategories.DisplayedBicats.Examples.FullSub}{fullsubbicat}]
  \label{ex:fullsub}
  Let $\B$ be a bicategory and $\pred : \B_0 \to \hProp$ a predicate on the 0-cells of $\B$.
  We define the \fat{full subbicategory of $\B$ with 0-cells satisfying $\pred$} as the bicategory whose objects are pairs $(a,p_a) : \Sum (x : \B_0) . \pred (x)$, 1-cells from $(a,p_a)$ to $(b,p_b)$ are 1-cells $a \onecell b$ in $\B$, and 2-cells are as in $\B$.
  In \Cref{ex:dispfullsub} we present a construction of this bicategory using displayed bicategories.
\end{defi}

\begin{example}[\coqident{Bicategories.Core.Examples.Groupoids}{grpds}]
\label{ex:grpd}
We define the bicategory $\Grpd$ as the full subbicategory of $\Cat$ in which every object is a groupoid.
\end{example}

For 1-categories the ``correct'' notion of equality is not isomorphism of categories, but equivalence of categories.
Similarly, the right notion of equality for bicategories is \emph{biequivalence}.
To talk about biequivalences we need to introduce \emph{pseudofunctors}.
\begin{defi}[\coqident{Bicategories.PseudoFunctors.PseudoFunctor}{psfunctor}]
\label{def:psfun}
Let $\B$ and $\C$ be bicategories.
A \fat{pseudofunctor} $F$ from $\B$ to $\C$ consists of
\begin{itemize}
	\item A function $F_0 : \B_0 \rightarrow \C_0$;
	\item For all $a, b : \B_0$, a function $F_1 : \B_1(a, b) \rightarrow \C_1(F_0(a), F_0(b))$;
	\item For all $f, g : \B_1(a, b)$, a function $F_2 : \B_2(f, g) \rightarrow \C_2(F_1(f), F_1(g))$;
	\item For each $a : \B_0$ an invertible 2-cell $\identitor{F}(a) : \id_1(F_0(a)) \mytwocell F_1(\id_1(a))$;
	\item For each $f : \B_1(a, b)$ and $g : \B_1(b, c)$, an invertible 2-cell $\compositor{F}(f, g) : F_1(f) \cdot F_1(g) \mytwocell F_1(f \cdot g)$
\end{itemize}
such that
\[
F_2(\id_2(f)) = \id_2(F_1(f)) \quad \quad F_2(f \vcomp g) = F_2(f) \vcomp F_2(g)
\]
and such that the following diagrams commute
(where all free variables should be taken to be universally quantified):
\[
\xymatrix@C=1.7em
{
	F_1(f) \cdot F_1(g_1) \ar@{=>}[d]_-{F_1(f) \whiskerl F_2(\tc)} \ar@{=>}[rr]^-{\compositor{F}(f,g_1)} & & F_1(f \cdot g_1) \ar@{=>}[d]^{F_2(f \whiskerl \tc)} \\
	F_1(f) \cdot F_1(g_2) \ar@{=>}[rr]_-{\compositor{F}(f,g_2)} & & F_1(f \cdot g_2)
}
\quad
\xymatrix@C=1.7em
{
	F_1(f_1) \cdot F_1(g) \ar@{=>}[d]_-{F_2(\tc) \whiskerr F_1(g)} \ar@{=>}[rr]^-{\compositor{F}(f_1,g)} & & F_1(f_1 \cdot g)) \ar@{=>}[d]^{F_2(\tc \whiskerr g)}\\
	F_1(f_2) \cdot F_1(g) \ar@{=>}[rr]_-{\compositor{F}(f_2,g)} & & F_1(f_2 \cdot g)
}
\]
\[
\xymatrix
{
	\id_1(F_0(a)) \cdot F_1(f) \ar@{=>}[rr]^-{\lunitor(F_1(f))} \ar@{=>}[d]_{\identitor{F}(a) \whiskerr F_1(f)} & &  F_1(f)\\
	F_1(\id_1(a)) \cdot F_1(f) \ar@{=>}[rr]_-{\compositor{F}(\id_1(a), f)} & &F_1(\id_1(a) \cdot f) \ar@{=>}[u]_-{F_2(\lunitor(f))}
}
\]
\[
\xymatrix
{
	F_1(f) \cdot \id_1(F_0(b)) \ar@{=>}[rr]^-{\runitor(F_1(f))} \ar@{=>}[d]_{F_1(f) \whiskerl \identitor{F}(b)} & &  F_1(f)\\
	F_1(f) \cdot F_1(\id_1(b)) \ar@{=>}[rr]_-{\compositor{F}(f, \id_1(b))} & &F_1(f \cdot \id_1(b)) \ar@{=>}[u]_-{F_2(\runitor(f))}
}
\]
\[
\xymatrix
{
	F_1(f) \cdot (F_1(g) \cdot F_1(h)) \ar@{=>}[rrr]^-{\assoc(F_1(f),F_1(g),F_1(h))} \ar@{=>}[d]_{F_1(f) \whiskerl \compositor{F}(g,h)}
		& & & (F_1(f) \cdot F_1(g)) \cdot F_1(h) \ar@{=>}[d]^{\compositor{F}(f,g) \whiskerr F_1(h)} \\
	F_1(f) \cdot F_1(g \cdot h) \ar@{=>}[d]_-{\compositor{F}(f,g \cdot h)} & & & F_1(f \cdot g) \cdot F_1(h) \ar@{=>}[d]^-{\compositor{F}(f \cdot g, h)}\\
	F_1(f \cdot (g \cdot h)) \ar@{=>}[rrr]_-{F_2(\assoc(f,g,h))} & & & F_1((f \cdot g) \cdot h)
}
\]
We write $\pseudoF{\B}{\C}$ for the type of pseudofunctors from $\B$ to $\C$.
\end{defi}

In the remainder of the paper, we sometimes write $F(a)$ instead of $F_0(a)$, and we use the same convention for $F_1$ and $F_2$. We call the 2-cells $\identitor{F}$ and $\compositor{F}$ the \emph{identitor} and \emph{compositor}, respectively.
From each pseudofunctor $F : \pseudoF{\B}{\C}$ we can assemble functors $F_1(a, b) : \homC{\B_1}{a}{b} \rightarrow \homC{\C_1}{F(a)}{F(b)}$ between the hom-categories.

\begin{defi}[\coqident{Bicategories.Transformations.PseudoTransformation}{pstrans}]
\label{def:pstrans}
Let $\B$ and $\C$ be bicategories and $F, G : \pseudoF{\B}{\C}$ pseudofunctors between them.
Then a \fat{pseudotransfomation} $\eta$ from $F$ to $G$ consists of
\begin{itemize}
	\item For each $a : \B_0$ a 1-cell $\eta_0(a) : F_0(a) \onecell G_0(a)$;
	\item For each $a, b : \B_0$ and $f : \B_1(a, b)$, an invertible 2-cell $\eta_1(f) : \eta_0 (a) \cdot G_1(f) \mytwocell F_1(g) \cdot \eta_0(b)$
\end{itemize}
such that the following diagrams commute
\[
\xymatrix
{
	\eta_0(a) \cdot \id_1 \twoar[r]^-{\runitor} \twoar[d]_{\eta_0(a) \whiskerl \identitor{G}(a)} & \eta_0(a) \twoar[r]^{\lunitor^{-1}} & {\id_1} \cdot \eta_0(a) \twoar[d]^{\identitor{F}(a) \whiskerr \eta_0(a)}\\
	\eta_0(a) \cdot G_1(\id_1) \twoar[rr]_-{\eta_1(\id_1(a))} & & F_1(\id_1) \cdot \eta_0(a)
}
\]
\[
\xymatrix
{
	\eta_0(a) \cdot (G_1(f) \cdot G_1(g)) \twoar[rr]^{\assoc} \twoar[dd]_{\eta_0 \whiskerl \compositor{G}}	& & (\eta_0(a) \cdot G_1(f)) \cdot G_1(g) \twoar[d]^{\eta_1(f) \whiskerr G_1(g)} \\
																																			 & & (F_1(f) \cdot \eta_0(b)) \cdot G_1(g) \twoar[d]^{\assoc^{-1}}\\
	\eta_0(a) \cdot G_1(f \cdot g) \twoar[dd]_{\eta_1(f \cdot g)}					 	  									 & & F_1(f) \cdot (\eta_0(b) \cdot G_1(g)) \twoar[d]^{F_1(f) \whiskerl \eta_1(g)}\\
																																			& & F_1(f) \cdot (F_1(g) \cdot \eta_0(c)) \twoar[d]^{\assoc}\\
	F_1(f \cdot g) \cdot \eta_0(c) 																								& & (F_1(f) \cdot F_1(g)) \cdot \eta_0(c) \twoar[ll]_{\compositor{F} \whiskerr \eta_0(c)}
}
\]
We write $\pstrans{F}{G}$ for the type of pseudotransformations from $F$ to $G$.
\end{defi}

\begin{defi}[\coqident{Bicategories.Modifications.Modification}{modification}]
\label{def:modif}
Let $\B$ and $\C$ be bicategories, $F, G : \pseudoF{\B}{\C}$ be pseudofunctors, and $\eta, \theta : \pstrans{F}{G}$ be pseudotransformations.
A \fat{modification} $\modvar$ from $\eta$ to $\theta$ consists of 2-cells $\modvar(a) : \eta(a) \mytwocell \theta(a)$ for each $a : \B$ such that
\[
\xymatrix
{
	\eta(a) \cdot G_1(f) \twoar[r]^{\eta(f)} \twoar[d]_{\modvar(a) \whiskerr G_1(f)} & F_1(f) \cdot \eta(b) \twoar[d]^{F_1(f) \whiskerl \modvar(b)}\\
	\theta(a) \cdot G_1(f) \twoar[r]_{\theta(f)} & F_1(f) \cdot \theta(b)
}
\]
commutes for any $a, b : \B$ and $f : \B_1(a, b)$.
We write $\modif{\eta}{\theta}$ for the type of modifications from $\eta$ to $\theta$.
\end{defi}

To illustrate these three definitions, we look at some examples.

\begin{example}
Let $X$ and $Y$ be 2-types.
\begin{itemize}
	\item (\coqident{Bicategories.PseudoFunctors.Examples.ApFunctor}{ap_psfunctor})
	Each function $f : X \rightarrow Y$ induces a pseudofunctor $\apPS{f} : \pseudoF{\fbg{X}}{\fbg{Y}}$,
	which sends objects $x : X$ to $f(x)$, 1-cells $p : x = y$ to $\ap{f}{p}$, and 2-cells $h : p = q$ to $\ap{(\constfont{ap} \ f)}{h}$.
	\item (\coqident{Bicategories.Transformations.Examples.ApTransformation}{ap_pstrans})
	Suppose we have $f, g : X \rightarrow Y$ and $e : \prod_{x : X} f(x) = g(x)$.
	Then we obtain a pseudotransformation $\apPT{e} : \pstrans{\apPS{f}}{\apPS{g}}$ whose component at $x$ is $e(x)$, and whose actions on 1-cells are given by path induction.
	\item (\coqident{Bicategories.Modifications.Examples.ApModification}{ap_modification})
	Let $f, g : X \rightarrow Y$ and $e_1, e_2 : \prod_{x : X} f(x) = g(x)$.
	Then each family of paths $h : \prod_{x : X} e_1(x) = e_2(x)$ gives rise to a modification $\apMOD{h} : \modif{\apPT{e_1}}{\apPT{e_2}}$ whose component at $x$ is $h(x)$.
\end{itemize}
\end{example}

\begin{example}
\label{ex:psfunpstrans}
We have the following pseudofunctors and pseudotransformations:
\begin{itemize}
	\item (\coqident{Bicategories.PseudoFunctors.Examples.Identity}{id_psfunctor})
	Given a bicategory $\B$, we have the identity pseudofunctor $\id(\B)$ from $\B$ to $\B$.
	Its action on 0-cells, 1-cells, and 2-cells is the identity.
	\item (\coqident{Bicategories.PseudoFunctors.Examples.Composition}{comp_psfunctor})
	Given bicategories $\B_1$, $\B_2$, and $\B_3$ and pseudofunctors $F : \pseudoF{\B_1}{\B_2}$ and $G : \pseudoF{\B_2}{\B_3}$, then we have a pseudofunctor $F \cdot G$ from $\B_1$ to $\B_3$.
        It sends objects $a$ to $G_0(F_0(a))$, 1-cells $f$ to $G_1(F_1(f))$, and 2-cells $\tc$ to $G_2(F_2(\tc))$.
	\item (\coqident{Bicategories.Transformations.PseudoTransformation}{id_pstrans})
	Given bicategories $\B_1$ and $\B_2$ and a pseudofunctor $F : \pseudoF{\B_1}{\B_2}$, we have a pseudotransformation $\id(F)$ from $F$ to $F$.
        It sends objects $a$ to $\id_1(F_1(a))$, and similarly for 1-cells.
	\item (\coqident{Bicategories.Transformations.PseudoTransformation}{comp_pstrans})
	Given bicategories $\B_1$ and $\B_2$, pseudofunctors $F, G, H : \pseudoF{\B_1}{\B_2}$, and two pseudotransformations $\theta_1 : \pstrans{F}{G}$ and $\theta_2 : \pstrans{G}{H}$, we have a pseudotransformation $\eta_1 \vcomp \eta_2 : \pstrans{F}{H}$.
        It sends objects $a$ to $\theta_1(a) \cdot \theta_2(a)$.
\end{itemize}
\end{example}

Note that we have a bicategory $\pseudo(\B, \C)$ of pseudofunctors, pseudotransformations, and modifications.
We construct this bicategory in \Cref{sec:pseudo} using displayed bicategories, and then we define invertible modifications to be invertible 2-cells in this bicategory.
With all this in place, we can define biequivalences.

\begin{defi}[\coqident{Bicategories.PseudoFunctors.Biequivalence}{biequivalence}]
\label{def:biequiv}
Let $\B$ and $\C$ be bicategories.
A \fat{biequivalence} from $\B$ to $\C$ consists of
\begin{itemize}
	\item A pseudofunctor $L : \pseudoF{\B}{\C}$;
	\item A pseudofunctor $R : \pseudoF{\C}{\B}$;
	\item Pseudotransformations $\eta : R \cdot L \mytwocell \id(\C)$ and $\eta_i : \id(\C) \mytwocell R \cdot L$;
	\item Pseudotransformations $\varepsilon : L \cdot R \mytwocell \id(\B)$ ad $\varepsilon_i : \id(\B) \mytwocell L \cdot R$;
	\item Invertible modifications
	\[
	m_1 : \modif{\eta \vcomp \eta_i}{\id} \quad \quad m_2 : \modif{\eta_i \vcomp \eta}{\id} \quad
	m_3 : \modif{\varepsilon \vcomp \varepsilon_i}{\id} \quad \quad m_4 : \modif{\varepsilon_i \vcomp \varepsilon}{\id}
	\]
\end{itemize}
\end{defi}

Usually, the notion of biequivalence is not sufficient, and instead \emph{biadjoint biequivalences} are used.
The latter notion has an extra requirement, namely that $L$ and $R$ form a pseudoadjunction \cite{lack2000coherent}.
Note that this is similar to the situation in types  (see, \eg \cite[Section~4]{hottbook}) and categories (see, \eg \cite[Section~IV.4]{mac2013categories}), where one also considers coherent notions of equivalence.
However, we restrict our attention to biequivalences, because every biequivalence can be refined to a biadjoint biequivalence \cite[Theorem 3.1]{gurski2011biequivalences}.

As an example, we construct a biequivalence between 1-types (\Cref{ex:one_types}) and univalent groupoids (\Cref{ex:grpd}).

\begin{example}[\coqident{Bicategories.PseudoFunctors.Examples.PathGroupoid}{biequiv_path_groupoid}]
\label{ex:biequiv-grpds}
We construct a biequivalence between 1-types and univalent groupoids.
We only show how the involved pseudofunctors are defined.
\begin{itemize}
	\item (\coqident{Bicategories.PseudoFunctors.Examples.PathGroupoid}{path_groupoid}) Define a pseudofunctor $\pathgroupoid : \pseudoF{\Onetypes}{\Grpd}$.
	It sends a 1-type $X$ to the groupoid $\pathgroupoid(X)$ whose objects are $X$ and morphisms from $x$ to $y$ are paths $x = y$.
	\item (\coqident{Bicategories.PseudoFunctors.Examples.PathGroupoid}{objects_of_grpd}) Define a pseudofunctor $\obs : \pseudoF{\Grpd}{\Onetypes}$.
	It sends a groupoid $G$ to the 1-type $\obs(G)$ whose inhabitants are objects of $G$.
	Note that this is a 1-truncated type, because $G$ is univalent. \qedhere
\end{itemize}
\end{example}

\section{Univalent Bicategories}
\label{sec:univalence}
Recall that a \onecategory $\C$ (called `precategory' in \cite{rezk_completion}) is called \emph{univalent}
if, for every two objects $a,b : \C_0$, the function $\idtoiso_{a,b} : (a = b) \to \Iso{a}{b}$
mapping the constant path to the identity isomorphism is an equivalence.
For bicategories, where we have one more layer of structure, univalence can be imposed both \emph{locally} and \emph{globally}.

\begin{defi}[\coqfile{Bicategories.Core}{Univalence}]
  \label{def:univalence}
  Univalence for bicategories is defined as follows:
\begin{enumerate}
\item \label{def-item:local-univalence}
  Let $a, b : \B_0$ and $f, g : \B_1(a,b)$ be objects and morphisms of $\B$;
  by path induction we define a function $\idtoiso^{2,1}_{f,g} : f = g \to \invtwocell(f,g)$ which sends $\idpath(f)$ to $\id_2(f)$.
  A bicategory $\B$ is \fat{locally univalent} if, for every two objects $a,b : \B_0$
  and two 1-cells $f,g : \B_1(a,b)$, the function $\idtoiso^{2,1}_{f,g}$ is an equivalence.
\item \label{def-item:global-univalence} Let $a, b : \B_0$ be objects of $\B$;
  using path induction we define $\idtoiso^{2,0}_{a,b} : a = b \to \AdjEquiv{a}{b}$ sending $\idpath(a)$ to $\id_1(a)$.
  A bicategory $\B$ is \fat{globally univalent} if, for every two objects $a,b : \B_0$,
  the canonical function $\idtoiso^{2,0}_{a,b}$ is an equivalence.
\item (\coqident{Bicategories.Core.Univalence}{is_univalent_2})
  We say that $\B$ is \fat{univalent} if $\B$ is both locally and globally univalent.
\end{enumerate}
\end{defi}

Local univalence can be characterized via the hom-categories.
More precisely, it is equivalent to all hom-categories being univalent.

\begin{proposition}[\coqident{Bicategories.Core.Univalence}{is_univalent_2_1_weq_local_univ}]
A bicategory $\B$ is locally univalent if and only if for every $a, b : \B_0$ the category $\homC{B}{a}{b}$ is univalent.
\end{proposition}

\begin{rem}
If $\B$ and $\C$ are locally univalent and $F$ is a pseudofunctor from $\B$ to $\C$, then the identity and compositions are preserved up to a path instead of just an invertible 2-cell. 
However, this does not mean such pseudofunctors should be considered as strict, because these are not paths between elements of a set.
\end{rem}

Univalent bicategories satisfy a variant of the elimination principle of path induction.
More precisely, there are two such principles: a local one for invertible 2-cells and a global one for adjoint equivalences.
We start with the induction principle associated to invertible 2-cells:

\begin{proposition}[\coqident{Bicategories.Core.Univalence}{J_2_1}]
\label{prop:J_local}
Let $\B$ be a locally univalent bicategory.
Given a type family $Y$ and a function $y$ with types
\[
Y : \Prod (a , b : \B_0) . \Prod (f , g : \B_1(a , b)) . \invtwocell(f,g) \to \U
\qquad \quad
y : \Prod (a , b : \B_0) . \Prod (f : \B_1(a , b)). Y (a,b,f,f,\id_2(f)),
\]
there is a function
\[
\J_{2,1} (Y,y) : \Prod (a , b : \B_0) . \Prod (f , g : \B_1(a , b)) . \Prod (\theta : \invtwocell(f,g)) . Y (a,b,f,g,\theta)
\]
such that $\J_{2,1}(Y,y,a , b , f , f , \id_2(f)) = y(a , b , f)$.
\end{proposition}
In particular, in order to prove a predicate over all invertible 2-cells in a given locally univalent bicategory, it suffices to prove it for all identity 2-cells.

Next, we present the induction principle associated to adjoint equivalences:
\begin{proposition}[\coqident{Bicategories.Core.Univalence}{J_2_0}]
\label{prop:J_global}
Let $\B$ be a globally univalent bicategory. Given a type family $Y$ and a function $y$ with types
\[
Y : \Prod (a , b : \B_0) . a \adjequiv b \to \U
\qquad \qquad \qquad
y : \Prod (a : \B_0) . Y (a,a,\id_1(a)),
\]
there is a function
\[
\J_{2,0} (Y,y) : \Prod (a , b : \B_0) . \Prod (f : a \adjequiv b) . Y (a,b,f)
\]
such that $\J_{2,0} (Y,y,a , a , \id_1(a)) = y(a)$.
\end{proposition}
In particular, in order to prove a predicate over all adjoint equivalences in a given globally univalent bicategory, it suffices to prove it for all identity 1-cells.
Notice that in both induction principles the computation rules hold only up to propositional equality.
Next, we present some usage examples of how to use \Cref{prop:J_local,prop:J_global}.
The constructions described in \Cref{ex:comp_adjequiv} and \Cref{ex:pseudo_preserve_adjequiv} work for arbitrary bicategories, not just globally/locally univalent ones. Nevertheless, these constructions are considerably simpler if the involved bicategories satisfy certain univalence assumptions. 

\begin{example}[\coqident{Bicategories.Core.Univalence}{comp_adjoint_equivalence}]
\label{ex:comp_adjequiv}  
In a globally univalent bicategory $\B$, sequential composition of
adjoint equivalences can be defined in a way that resembles the
construction of composition of paths.  Consider the type family $Y (a
, b , f) \defeq \Prod (c : \B_0) . b \adjequiv c \to a \adjequiv c$
and the function $y(a) \defeq \Lam (c : \B_0) (f : a \adjequiv c)
. f$.  The composition of $f : a \adjequiv b$ and $g : b \adjequiv c$
is given by
\[
f \cdot_{\adjequiv} g \defeq \J_{2,0}(Y,y,a,b,f,c,g). 
\]
\end{example}
\begin{example}[\coqident{Bicategories.Core.Univalence}{left_adjequiv_invertible_2cell}]
\label{ex:adjequiv_invtwocell}  
Let $\B$ be a bicategory, $f,g : \B_1(a,b)$ and $\theta :
\invtwocell(f,g)$. If $f$ is an adjoint equivalence, then $g$ is an
adjoint equivalence as well.  While this result generally holds in any
bicategory $\B$, it is particularly simple to prove when $\B$ is
locally univalent. Applying \Cref{prop:J_local}, we are
left to prove the statement with $\theta$ as the identity 2-cell.
In that statement, $f$ and $g$ are definitionally equal, and hence the statement is trivially true.
\end{example}
\begin{proposition}
Every pseudofunctor $F : \pseudoF{\B}{\C}$ preserves adjoint
equivalences, that is, if $f : a \adjequiv b$ in $\B$, then $F_1(f) :
F_0(a) \adjequiv F_0(b)$ in $\C$.
\end{proposition}
\begin{proof}
  Lengthy but straightforward.
\end{proof}
If $\B$ is globally univalent and $\C$ is locally univalent, the above statement can be proved very easily.
\begin{proposition}[\coqident{Bicategories.PseudoFunctors.PseudoFunctor}{psfunctor_preserves_adjequiv}]
\label{ex:pseudo_preserve_adjequiv}  
If $\B$ is globally univalent and $\C$ is locally univalent, then every pseudofunctor $F : \pseudoF{\B}{\C}$ preserves adjoint
equivalences.
\end{proposition}
\begin{proof}
Applying \Cref{prop:J_global} on $f$, we are left
to prove that $F_1(\id_1(a))$ is an adjoint equivalence.  Since $F$ is
a pseudofunctor, there exists an invertible 2-cell $\identitor{F}(a) :
\id_1(F_0(a)) \mytwocell F_1(\id_1(a))$. Therefore, by 
\Cref{ex:adjequiv_invtwocell} and the fact that $\id_1(F_0(a))$ is an
adjoint equivalence, we conclude that $F_1(\id_1(a))$ is an adjoint
equivalence as well.
\end{proof}

Another consequence is that biequivalences between univalent bicategories gives rise to equivalences on the level of objects.

\begin{proposition}[\coqident{Bicategories.PseudoFunctors.Biequivalence}{biequivalence_to_object_equivalence}]
Given univalent bicategories $\B$ and $\C$, and a biequivalence $F$ from $\B$ to $\C$, then we get an equivalence of types $F_0 : \B_0 \simeq \C_0$.
\end{proposition}

While right adjoints are only unique up to isomorphism in general, they are unique up to identity if the bicategory is locally univalent:
\begin{proposition}[\coqident{Bicategories.Core.AdjointUnique}{isaprop_left_adjoint_equivalence}]
\label{prop:adjunique}
Let $\B$ be locally univalent.
Then having an adjoint equivalence structure on a 1-cell in $\B$ is a proposition.
\end{proposition}
As a consequence of this proposition we get the following:
\begin{theorem}
 In a univalent bicategory $\B$,
\begin{itemize}
 \item (\coqident{Bicategories.Core.AdjointUnique}{univalent_bicategory_0_cell_hlevel_4})
       the type $\B_0$ of 0-cells is a 2-type.
 \item (\coqident{Bicategories.Core.Univalence}{univalent_bicategory_1_cell_hlevel_3})
       for any two objects $a, b : \B_0$, the type $a \onecell b$ of 1-cells from $a$ to $b$ is a 1-type.
\end{itemize}
\end{theorem}

\Cref{prop:adjunique} has another important use:
to prove global univalence of a bicategory, we need to show that $\idtoiso^{2,0}_{a,b}$ is an equivalence.
Often we do that by constructing a function in the other direction and showing these two are inverses.
This requires comparing adjoint equivalences, which is done with the help of \Cref{prop:adjunique}.

Local univalence is also relevant when one discusses bicategorical analogues of limits and colimits.
To exemplify this, we look at biinitial objects, and we note that a similar discussion can be given for bifinal objects (\coqident{Bicategories.Colimits.Initial}{bifinal_unique}).
We start by defining \emph{biinitiality structures}.

\begin{defi}[\coqident{Bicategories.Colimits.Initial}{is_biinitial}]
\label{def:biinitial}
Let $\B$ be a bicategory and let $a$ be an object in $\B$.
Then a \fat{biinitiality structure} on $a$ consists of an external adjoint equivalence structure on the canonical functor from $\homC{B}{a}{b}$ to the unit category for each $b : \B$.
A \fat{biinitial object} is an object $a : \B$ together with a biinitiality structure on $a$.
\end{defi}

In general, adjoint equivalence structures are not necessarily unique, but they are if the bicategory is locally univalent.
As such, having a biinitiality structure is not necessarily a proposition, and instead, it should be viewed as a structure on the objects.
If the bicategory is locally univalent, however, then we can use \Cref{prop:adjunique} to show that biinitiality structures form a proposition.

\begin{proposition}[\coqident{Bicategories.Colimits.Initial}{isaprop_is_biinitial}]
Let $\B$ be a locally univalent bicategory.
Then for each $a : \B$ the type of biinitiality structures on $a$ is a proposition.
\end{proposition}

While local univalence affects the uniqueness of biinitiality structures, global univalence affects the uniqueness of biinitial objects.
Since limits and colimits are unique up adjoint equivalence, the type of biinitial objects is a proposition if the bicategory is univalent.

\begin{proposition}[\coqident{Bicategories.Colimits.Initial}{biinitial_unique}]
Let $\B$ be a univalent bicategory.
Then the type of biinitial objects in $\B$ is a proposition.
\end{proposition}

Before we discuss examples of biinitial objects, we give an equivalent definition of biinitiality formulated using universal mapping properties.

\begin{lemma}[\coqident{Bicategories.Colimits.Initial}{biinitial_weq_biinitial'}]
Let $\B$ be a bicategory and let $a$ be an object in $\B$.
Then $a$ has a biinitiality structure if and only if the following holds:
\begin{itemize}
	\item for every $b$ there is a 1-cell $a \onecell b$;
	\item for every two 1-cells $f, g : a \onecell b$ there is a unique 2-cell $f \mytwocell g$.
\end{itemize}
\end{lemma}

\begin{example}
Note that both $\Onetypes$ and $\Cat$ have a biinitial object.
\begin{itemize}
	\item (\coqident{Bicategories.Colimits.Initial}{biinitial_1_types}) The empty type is a biinitial object in $\Onetypes$.
	\item (\coqident{Bicategories.Colimits.Initial}{biinitial_cats}) The empty category is a biinitial object in $\Cat$.
\end{itemize}
\end{example}

Now let us prove that some examples from \Cref{sec:bicategories} are univalent.
\begin{example}
  The following bicategories are univalent:
  \label{ex:univalent}
  \begin{enumerate}
  \item (\coqfile{Bicategories.Core.Examples}{TwoType},
    \Cref{ex:bigroupoid} cont'd) The fundamental bigroupoid of each 2-type is univalent.
  \item (\coqfile{Bicategories.Core.Examples}{OneTypes}, \Cref{ex:one_types} cont'd)
    \label{ex:one_types:univalent}
 The bicategory of 1-types of a universe $\U$ is locally univalent;
 this is a consequence of function extensionality.
 If we assume the univalence axiom for $\U$, then 1-types form a univalent bicategory.
 To show that, we factor $\idtoiso^{2,0}$ as follows.
 \[
 \xymatrix
 {
 	X = Y \ar[rr]^-{\idtoiso^{2,0}_{X,Y}} \ar[dr]_-{\simeq} & & \AdjEquiv{X}{Y} \\
 	& X \simeq Y \ar[ur]_-{\simeq} &
 }
 \]
 The left function is an equivalence by univalence, and the right function is an equivalence by the characterization of adjoint equivalences in \Cref{ex:one_types}.
 The fact that this diagram commutes follows from \Cref{prop:adjunique}.
 \item (\coqfile{Bicategories.DisplayedBicats.Examples}{FullSub}, If $\B$ is univalent and $P$ is a predicate on $\B$, then so is the full subbicategory of $\B$ with those objects satisfying $P$. \label{ex:fullsub:univalent}
  \end{enumerate}
\end{example}
It is more difficult to prove that the bicategory of univalent categories is univalent, and we only give a brief sketch of this proof.
\begin{proposition}[\coqfile{Bicategories.Core.Examples}{BicatOfUnivCats}, \Cref{ex:cat} cont'd]
 The bicategory $\Cat$ is univalent.
\end{proposition}
Local univalence follows from the fact that the functor category $[C,D]$ is univalent if $D$ is.
For global univalence, we use that the type of identities on categories is equivalent to the type of adjoint equivalences between categories \cite[Theorem~6.17]{rezk_completion}.
The proof proceeds by factoring $\idtoiso^{2,0}$ as a chain of equivalences $(C = D) \xrightarrow{\sim} \CatIso{C}{D} \xrightarrow{\sim} \AdjEquiv{C}{D}$.
To our knowledge, a proof of global univalence was first computer-formalized by Rafa\"el Bocquet%
\footnote{\url{https://github.com/mortberg/cubicaltt/blob/master/examples/category.ctt}}.

In the previous examples, we proved univalence directly.
However, in many complicated bicategories such proofs are not feasible.
An example of such a bicategory is the bicategory $\pseudo(\B,\C)$ of pseudofunctors from $\B$ to $\C$, pseudotransformations, and modifications \cite{leinster:basic-bicats} (for a univalent bicategory $\C$).
Even in the 1-categorical case, proving the univalence of the category $[C, D]$ of functors from $C$ to $D$, and natural transformations between them, is tedious.
In \Cref{sec:disp_univalence}, we develop some machinery to prove the following theorem.
\begin{theorem}[\coqident{Bicategories.PseudoFunctors.Display.PseudoFunctorBicat}{psfunctor_bicat_is_univalent_2}]
  \label{thm:psfunct_univalent_2}
  If $\B$ is a (not necessarily univalent) bicategory and $\C$ is a univalent bicategory, then the bicategory $\pseudo(\B,\C)$ of pseudofunctors from $\B$ to $\C$ is univalent.
\end{theorem}

\section{Bicategories and 2-Categories}
\label{sec:2-categories}
In this section, we propose a definition of 2-category, and compare 2-categories to bicategories.
We start by defining strict bicategories.

\begin{defi}[\coqident{Bicategories.Core.Strictness}{locally_strict},\coqident{Bicategories.Core.Strictness}{is_strict_bicat}]
\label{def:1-strict}
A bicategory is called \fat{locally strict} if each $\B_1(x,y)$ is a set.
A \fat{1-strict} bicategory is a locally strict bicategory such that
\begin{enumerate}
	\item \label{item:lpunitor}
	for each $a, b : \B$ and $f : a \onecell b$ we have $\plunitor(f) : \id_1(a) \cdot f = f$, and $\idtoiso^{2,1}(\plunitor(f)) = \lunitor(f)$;
	\item \label{item:rpunitor}
	for each $a, b : \B$ and $f : a \onecell b$ we have $\prunitor(f) : f \cdot \id_1(b) = f$, and $\idtoiso^{2,1}(\prunitor(f)) = \runitor(f)$;
	\item \label{item:passociator}
	for each $a, b, c, d : \B$ and $f : a \onecell b, g : b \onecell c$, and $h : c \onecell d$ we have $\passoc(f,g,h) : f \cdot (g \cdot h) = (f \cdot g) \cdot h$, and $\idtoiso^{2,1}(\passoc(f,g,h)) = \assoc(f,g,h)$.
\end{enumerate}
\end{defi}

\begin{proposition}[\coqident{Bicategories.Core.Strictness}{isaprop_is_strict_bicat}]
Being a 1-strict bicategory is a proposition.
\end{proposition}

Now let us look at an example of a 1-strict bicategory.

\begin{example}[\coqident{Bicategories.Core.Examples.StrictCats}{strict_bicat_of_strict_cats}]
Recall that a category is called \emph{strict} if its objects form a set.
Define $\Cat_S$ to be the bicategory whose objects are strict categories, 1-cells are functors, and 2-cells are natural transformations.
Then $\Cat_S$ is a 1-strict bicategory.
\end{example}

The bicategory $\Cat$ of univalent categories is not 1-strict.
This is because functors between two categories do not necessarily form a set.

\begin{proposition}[\coqident{Bicategories.Core.Strictness}{cat_not_a_two_cat}]
Assuming the univalence axiom, we can show that the bicategory $\Cat$ is not 1-strict.
\end{proposition}

\begin{rem}
 Without the local strictness condition (the 1-cells form sets), the conditions of \Cref{item:lpunitor,item:rpunitor,item:passociator} of \Cref{def:1-strict} are not well-behaved.
 In our \UniMath formalization, we study a coherent version of \Cref{def:1-strict} without the requirement that the 1-cells form a set, under the name of \fat{coherent strictness structures} (\coqident{Bicategories.Core.Strictness}{coh_strictness_structure}).
 One can show that the additional coherence conditions are unique and automatic when the bicategory under consideration is locally strict or locally univalent.
 Furthermore, when the bicategory is locally univalent, the type of coherent strictness structures on it is contractible (\coqident{Bicategories.Core.Strictness}{unique_strictness_structure_is_univalent_2_1}).
\end{rem}

Next we look at 2-categories.
These are defined as 1-categories with additional structure and properties.

\begin{defi}[\coqident{CategoryTheory.Core.TwoCategories}{two_cat}]
A \fat{2-category} $\C$ consists of
\begin{itemize}
	\item a category $\C_0$;
	\item for each $x, y : \C_0$ and $f, g : x \onecell y$ a set $\C_2(f, g)$ of \fat{2-cells};
	\item an \fat{identity 2-cell} $\id_2(f) : \C_2(f,f)$;
	\item a \fat{vertical composition} $\tc \vcomp \tcB : \C_2(f,h)$ for all 1-cells $f,g,h : \C_1(a,b)$ and 2-cells $\tc : \C_2(f,g)$ and $\tcB : \C_2(g,h)$;
	\item a \fat{left whiskering} $f \whiskerl \tc : \C_2(f \cdot g, f \cdot h)$ for all 1-cells $f : \C_1(a,b)$ and $g,h : \C_1(b,c)$ and 2-cells $\tc : \C_2(g,h)$;
	\item a \fat{right whiskering} $\tc \whiskerr h : \C_2(f \cdot h, g \cdot h)$ for all 1-cells $f, g : \C_1(a,b)$ and $h : \C_1(b,c)$ and 2-cells $\tc : \C_2(f,g)$;
\end{itemize}
such that, for all suitable objects, 1-cells, and 2-cells,
\begin{itemize}
\item $\id_2(f) \vcomp \tc = \tc, \quad \tc \vcomp \id_2(g) = \tc, \quad \tc \vcomp (\tcB \vcomp \tcC) = (\tc \vcomp \tcB) \vcomp \tcC$;
\item $f \whiskerl (\id_2 g) = \id_2(f \cdot g), \quad f \whiskerl (\tc \vcomp \tcB) = (f \whiskerl \tc) \vcomp (f \whiskerl \tcB)$;
\item $(\id_2 f) \whiskerr g = \id_2(f \cdot g), \quad (\tc \vcomp \tcB) \whiskerr g = (\tc \whiskerr g) \vcomp (\tcB \whiskerr g)$;
\item 
$(\id_1 a \whiskerl x) \vcomp \idtotwomor (\idleft(g)) = \idtotwomor (\idleft(f)) \vcomp x$;
\item 
$(x \whiskerr \id_1 b) \vcomp \idtotwomor(\idright(g)) = \idtotwomor(\idright(f)) \vcomp x$;
\item 
$(f \whiskerl (g \whiskerl x)) \vcomp \idtotwomor(\assoceq(f,g,i)) = \idtotwomor(\assoceq(f,g,h)) \vcomp (f \cdot g \whiskerl x)$;
\item 
$f \whiskerl (x \whiskerr i) \vcomp \idtotwomor(\assoceq (f,h,i)) = \idtotwomor (\assoceq(f,g,i)) \vcomp ((f \whiskerl x) \whiskerr i)$;
\item 
$ \idtotwomor(\assoceq(f,h,i)) \vcomp (x \whiskerr h \whiskerr i) = (x \whiskerr h \cdot i) \vcomp \idtotwomor(\assoceq(g,h,i))$.
\end{itemize}
\end{defi}
Here, the function $\idtotwomor_{f,g} : (f = g) \to (f \mytwocell g)$ is defined by path induction, sending the identity path to the identity 2-cell.
The paths $\idleft(f)$, $\idright(g)$, and $\assoceq(f,g,h)$ are those given by the categorical axioms for $\C_0$.

We call 0-cells of a 2-category $\C$ the objects of $\C_0$, and 1-cells the morphisms of the category $\C_0$.
In particular, the 1-cells between every two 0-cells of a 2-category always form a set.

\begin{rem}
 The last few axioms of a 2-category could, equivalently, be stated using transport along a categorical equality axiom (e.g., along $\idleft(f)$), instead of using $\idtotwomor$.
\end{rem}

The type of 1-strict bicategories is equivalent to that of 2-categories.

\begin{problem}
\label{prob:strict-to-2cat}
To construct an equivalence between the type of 1-strict bicategories and the type of 2-categories.
\end{problem}

\begin{construction}[\coqident{Bicategories.Core.Strictness}{strict_bicat_to_two_cat}]{prob:strict-to-2cat}
In one direction, suppose $\C$ is a 2-category.
We associate to $\C$ the following bicategory:
\begin{enumerate}
\item 0-cells, 1-cells, and 2-cells are those of $\C$;
\item composition and identity of 1-cells and 2-cells are those of $\C$, respectively;
\item whiskering is given by the whiskering of $\C$;
\item left and right unitors, and associators, are 2-cells induced by the corresponding equality axioms via $\idtotwomor$.
\end{enumerate}
The bicategorical axioms are then easily shown, using compatibility, in a suitable sense, of $\idtotwomor$ with composition of paths (which corresponds to composition of 2-cells) and functions on paths (which corresponds to whiskering).
The resulting bicategory is 1-strict.

In the other direction, suppose $\B$ is a 1-strict bicategory.
We associate the following 2-category to $\B$:
\begin{enumerate}
\item 0-cells, 1-cells, and 2-cells are those of $\B$, respectively;
\item composition, identities, and whiskering are given by the corresponding operations of $\B$;
\item the equality axioms for composition of 1-cells are proved using the strictness properties of $\B$;
\item the remaining axioms are proved using suitable compatibility results about $\idtotwomor$.
\end{enumerate}

\noindent
The two functions are easily shown to be inverse to each other, thus forming an equivalence of types.
\end{construction}

\section{The Yoneda Embedding}
\label{sec:yoneda}
In this section, we show that any locally univalent bicategory naturally embeds into a univalent one, via the Yoneda embedding.
This construction is similar to the Rezk completion for categories \cite[Theorem 8.5]{rezk_completion} and it makes use of the Yoneda lemma.
We start by discussing representable pseudofunctors, pseudotransformations, and modifications.
These are used to define the desired embedding.

\begin{defi}[Representables]
Let $\B$ be a locally univalent bicategory.
\begin{itemize}
	\item (\coqident{Bicategories.PseudoFunctors.Representable}{representable}) Given an object $a : \B$, we define the \fat{representable pseudofunctor} $\repps(a)$ from $\op{\B}$ (see \Cref{ex:op}) to $\Cat$.
	It sends objects $b$ to the category $\homC{\B_1}{b}{a}$ and 1-cells $f : b_1 \onecell b_2$ to the functor $\repps(a)(f) : \homC{\B_1}{b_2}{a} \to \homC{\B_1}{b_1}{a}$ given by $g \mapsto f \cdot g$. If we have 1-cells $f, g : b_1 \onecell b_2$ and a 2-cell $\tc : f \mytwocell g$, then $\repps(a)(\tc) : \repps(a)(f) \mytwocell \repps(a)(g)$ is the natural transformation whose component for each $h : b_2 \onecell a$ is $\tc \whiskerr h$.
	\item (\coqident{Bicategories.PseudoFunctors.Representable}{representable1}) Let $a, b : \B$ be objects and let $f : a \onecell b$ be a 1-cell.
	Then we define the \fat{representable pseudotransformation} $\reptr(f)$ from $\repps(a)$ to $\repps(b)$.
	Its component for each $c : \B$ is the functor $\reptr(f)(c) : \homC{\B_1}{c}{a} \onecell \homC{\B_1}{c}{b}$ sending $g$ to $g \cdot f$.
	If we have $c_1, c_2 : \B$ and a 1-cell $g : c_1 \onecell c_2$, then the naturality 2-cell $\reptr(f)(g) : \reptr(f)(c_1) \cdot \repps(b)(g) \mytwocell \repps(a)(g) \cdot \reptr(f)(c_2)$ is a natural transformation, whose component for each $h$ is $\assoc(g, h, f) : g \cdot (h \cdot f) \mytwocell (g \cdot h) \cdot f$.
	\item (\coqident{Bicategories.PseudoFunctors.Representable}{representable2}) Suppose that we have 0-cells $a, b : \B$, 1-cells $f, g : a \onecell b$, and a 2-cell $\tc : f \mytwocell g$.
	Then the \fat{representable modification} $\repmo(\alpha)$ from $\reptr(f)$ to $\reptr(g)$ is a modification, whose component for each $c : \B$ is the natural transformation defined on $h : \B(c,a)$ by $h \whiskerl \tc$.
\end{itemize}
\end{defi}

\begin{defi}[\coqident{Bicategories.PseudoFunctors.Yoneda}{y}]
Let $\B$ be a locally univalent bicategory.
Then the \fat{Yoneda embedding} $\yoneda : \pseudoF{\B}{\pseudo(\op{\B}, \Cat)}$ is defined as
\begin{align*}
\yoneda(a) &= \repps(a) && \mbox{for } a : \B \\
\yoneda(f) &= \reptr(f)  && \mbox{for } a, b : \B, f : a \onecell b \\
\yoneda(\tc) &= \repmo(\tc) && \mbox{for } a, b : \B, f, g : a \onecell b, \theta : f \mytwocell b
\end{align*}
\end{defi}

\begin{problem}[Bicategorical Yoneda lemma]
\label{prob:yoneda}
Given a locally univalent bicategory $\B$, a pseudofunctor  $P : \pseudoF{\op{\B}}{\Cat}$, and $a : \B$, to construct an adjoint equivalence between the categories $\homC{\pseudo(\op{\B}, \Cat)}{\yoneda(a)}{P}$ and $P(a)$.
\end{problem}

\begin{construction}[\coqident{Bicategories.Core.YonedaLemma}{bicategorical_yoneda_lemma}]{prob:yoneda}
To construct this, we provide
\begin{itemize}
	\item (\coqident{Bicategories.Core.YonedaLemma}{yoneda_to_presheaf}) A functor $F$ from $\pstrans{\yoneda(a)}{P}$ to $P(a)$;
	\item (\coqident{Bicategories.Core.YonedaLemma}{presheaf_to_yoneda}) A functor $G$ from $P(a)$ to $\pstrans{\yoneda(a)}{P}$;
	\item (\coqident{Bicategories.Core.YonedaLemma}{yoneda_unit}) A natural isomorphism from the identity to $F \cdot G$;
	\item (\coqident{Bicategories.Core.YonedaLemma}{yoneda_counit}) A natural isomorphism from $G \cdot F$ to the identity.
\end{itemize}
We only discuss the data of the involved functors.
The functor $F$ sends pseudotransformations $\tau$ to $\tau(a)(\id_1(a))$ and modifications $m$ to $m(a)(\id_1(a))(a)$.
In the other direction, $G$ sends objects $z : P(a)$ to the pseudotransformation whose components are $P(f)(z)$ with $b : \op{\B}$ and $f : b \onecell a$.
\end{construction}

Now let us use the bicategorical Yoneda lemma to construct for each locally univalent bicategory a weakly equivalent univalent bicategory.
We follow the construction of the Rezk completion by Ahrens, Kapulkin, and Shulman \cite{rezk_completion}, and take the image of the Yoneda embedding to be the univalent completion.

First, we define weak equivalences of bicategories.
\begin{defi}
	\label{def:weakequiv}
	Let $\B$ and $\C$ be bicategories and let $F : \pseudoF{\B}{\C}$ be a pseudofunctor. We say
	\begin{itemize}
		\item (\coqident{Bicategories.PseudoFunctors.PseudoFunctor}{local_equivalence}) $F$ is a \fat{local equivalence} if for each $x, y : \B$ the functor from $\homC{\B_1}{x}{y}$ to $\homC{\C_1}{F(x)}{F(y)}$ induced by $F$ is an adjoint equivalence.
		\item (\coqident{Bicategories.PseudoFunctors.PseudoFunctor}{essentially_surjective}) $F$ is \fat{essentially surjective} if for each $y : \C$ there merely exists an $x : \B$ and an adjoint equivalence from $F(x)$ to $y$.
		\item (\coqident{Bicategories.PseudoFunctors.PseudoFunctor}{weak_equivalence}) $F$ is a \fat{weak equivalence} if $F$ is both a local equivalence and essentially surjective.
	\end{itemize}
\end{defi}

The notion of weak equivalence has already been studied in classical mathematics where, using the axiom of choice, it was shown to be equivalent to the usual notion of equivalence \cite{lack20102,leinster:basic-bicats}.
However, these notions are generally not equivalent in a constructive setting, but we conjecture that they are for univalent bicategories.

Furthermore, the notion of weak equivalence can be weakened by requiring that the pseudofunctor only induces a weak equivalence of categories on the hom-categories.
Such a weaker notion would be useful if one desires to find a univalent completion of arbitrary bicategories instead of just locally univalent ones.
To do so, we anticipate a two-step process: first, a local completion, which embeds bicategories in locally univalent ones, followed by, second, the construction described in this section.
More concretely, for any bicategory $\B$ we expect to be able to construct pseudofunctors as in the following diagram.
\[
\xymatrix
{
	\B \ar[r]^-{\etaloc} & \loccompletion{B} \ar[r]^-{\etaglob} & \completion{B}
}
\]
Here, $\loccompletion{B}$ is a locally univalent bicategory, and $\completion{B}$ is a univalent one.
While the pseudofunctor $\etaglob : \loccompletion{B} \to \completion{B}$ would be a weak equivalence according to \Cref{def:weakequiv}, $\etaloc : \B \to \loccompletion{B}$, would not be: locally, it consists of \emph{weak} equivalences of categories instead of equivalences.
Hence, the more general notion would thus be applicable if one is interested in the univalent completion of arbitrary bicategories.
We expect such a local completion $\etaloc : B \to \loccompletion{B}$ can be constructed by taking the Rezk completion of every hom-category, and since that only yields a weak equivalence of categories, the resulting pseudofunctor is not locally an equivalence.
However, we only consider the second step in this paper, and leave the construction of $\etaloc : \B \to \loccompletion{B}$ as an open problem.

\begin{conjecture}
For every bicategory $\B$ there is a locally univalent bicategory $\loccompletion{B}$ and a pseudofunctor $\etaloc : \pseudoF{\B}{\loccompletion{B}}$ which is essentially surjective and locally a weak equivalence of categories.
\end{conjecture}

Weak equivalences between univalent categories are actually equivalences~\cite[Lemma~6.8]{rezk_completion}.
We conjecture that the same is possible for bicategories.

\begin{conjecture}
Every weak equivalence between univalent bicategories is a biequivalence.
\end{conjecture}

From the Yoneda lemma we know that $\yoneda$ is a local equivalence:
\begin{corollary}[\coqident{Bicategories.Core.YonedaLemma}{yoneda_mor_is_equivalence}]
	\label{cor:local_equiv}
The pseudofunctor $\yoneda$ is a local equivalence.
\end{corollary}
However, $\yoneda$ is not essentially surjective: the bicategory $\pseudo(\op{\B}, \Cat)$ contains non-representable presheaves.
To make $\yoneda$ essentially surjective we restrict the bicategory of presheaves to the full image of the Yoneda embedding.
\begin{defi}[\coqident{Bicategories.Core.Examples.Image}{full_image}]
Let $\B$ and $\C$ be bicategories and let $F : \pseudoF{\B}{\C}$ be a pseudofunctor.
Then the \fat{full image} $\image(F)$ of $F$ is the full subbicategory consisting of those objects $c$ in $\C$ for which there merely exists $b : \B$ such that $F(b) = c$.
\end{defi}

\begin{proposition}[\coqident{Bicategories.Core.Examples.Image}{is_univalent_2_full_image}]
\label{prop:univ_image}
If $\C$ is univalent, then so is the full image of $F : \pseudoF{\B}{\C}$.
\end{proposition}
\begin{proof}
  Follows from \Cref{ex:fullsub:univalent} in \Cref{ex:univalent}.
\end{proof}
\begin{defi}[\coqident{Bicategories.PseudoFunctors.Examples.CorestrictImage}{corestrict_full_image}]
Again let $\B$ and $\C$ be bicategories and suppose we have a pseudofunctor $F : \pseudoF{\B}{\C}$.
Then we define the \fat{corestriction} of $F$ to be the pseudofunctor $\restrict{F} : \pseudoF{\B}{\image(F)}$ which sends $b$ to $F(b)$.
The fact that $F(b)$ is indeed in the image is witnessed by $|(b , \idpath)|$.
\end{defi}

Now everything is in place to construct the desired embedding into a univalent bicategory.
\begin{problem}
\label{prob:globalcomp}
For each locally univalent bicategory $\B$, to construct a univalent bicategory $\completion{\B}$ and a weak equivalence $F : \pseudoF{\B}{\completion{\B}}$.
\end{problem}

\begin{construction}[\coqident{Bicategories.Core.YonedaLemma}{rezk_completion_2_0}]{prob:globalcomp}
\label{constr:rezk-emb}
We define $\completion{\B}$ to be the image of the Yoneda embedding $\yoneda : \pseudoF{\B}{\pseudo(\op{\B}, \Cat)}$.
Since the codomain of $\yoneda$ is univalent by \Cref{thm:psfunct_univalent_2}, the image is univalent as well by \Cref{prop:univ_image}.
Note that the corestriction gives rise to a pseudofunctor $\restrict{\yoneda} : \pseudoF{\B}{\completion{\B}}$.
It is essentially surjective by construction.
Furthermore, $\yoneda$ is a local equivalence by \Cref{cor:local_equiv}, and local equivalences are preserved by corestriction.
Hence, $\restrict{\yoneda}$ is indeed a weak equivalence.
\end{construction}
Note that \Cref{constr:rezk-emb} raises universe levels: the bicategory $\completion{\B}$ lives in a higher universe than $B$ itself, for the same reasons as in the 1-categorical case; see \cite[Remark~8.6]{rezk_completion}.

\section{Displayed Bicategories}
\label{sec:displayed}
Now let us study how to construct more complicated univalent bicategories.
To that end, we introduce \emph{displayed bicategories}, the bicategorical analog to the notion of
displayed category developed in \cite{AhrensL19}.
A displayed \onecategory $D$ over a given (base) category $C$ consists of a family of objects over objects in $C$
and a family of morphisms over morphisms in $C$
together with suitable displayed operations of composition and identity.
A category $\total{D}$ is then constructed, the objects and morphisms of which are pairs of objects and morphisms from $C$
and $D$, respectively. 
Properties of $\total{D}$, in particular univalence, can be shown from analogous, but simpler, conditions on $C$ and $D$.

A prototypical example is the following displayed category over $C\defeq \Set$: an object over a set $X$ is a group structure on $X$,
and a morphism over a function $f : X \to X'$ from group structure $G$ (on $X$) to group structure $G'$ (on $X'$) 
is a proof of the fact that $f$ is compatible with $G$ and $G'$.
The total category is the category of groups, and its univalence follows from univalence of $\Set$ and a univalence property 
of the displayed data.

Just like in 1-category theory, many examples of bicategories are obtained by endowing previously considered bicategories with additional structure.
An example is the bicategory of pointed 1-types in $\U$.
The objects in this bicategory are pairs of a 1-type $A$ and an inhabitant $a : A$.
The morphisms are pairs of a morphism $f$ of 1-types and a path witnessing that $f$ preserves the selected points.
Similarly, the 2-cells are pairs of a homotopy $p$ and a proof that this $p$ commutes with the point preservation proofs.
Thus, this bicategory is obtained from $\Onetypes_{\U}$ by endowing the cells on each level with additional structure.

Of course, the structure should be added in such a way that we are guaranteed to obtain a bicategory at the end.
Now let us give the formal definition of displayed bicategories.

\begin{defi}[\coqident{Bicategories.DisplayedBicats.DispBicat}{disp_bicat}]
  \label{def:disp_bicat}
Given a bicategory $\B$, a \fat{displayed bicategory $\D$ over $\B$} is given by data analogous to that of a bicategory, to which the numbering refers:
\begin{description}
	\item [\itemref{item:0-cell}] for each $a : \B_0$ a type $\dob{\D}{a}$ of displayed 0-cells over $a$;
	\item [\itemref{item:1-cell}] for each $f : a \onecell b$ in $\B$ and $\aa : \dob{\D}{a} , \bb : \dob{\D}{b}$ a type $\dmor{\aa}{\bb}{f}$ of displayed 1-cells over $f$;
	\item [\itemref{item:2-cell}] for each $\tc : f \mytwocell g$ in $\B$, $\ff : \dmor{\aa}{\bb}{f}$ and $\gg : \dmor{\aa}{\bb}{g}$ a set $\dtwo{\ff}{\gg}{\tc}$ of displayed 2-cells over $\tc$
\end{description}
and dependent versions of operations and laws from \Cref{def:bicat}, which are
\begin{description}
	\item [\itemref{item:id-1}] for each $a : \B_0$ and $\aa : \dob{\D}{a}$, we have $\did_1(\aa) : \dmor{\aa}{\aa}{\id_1(a)}$;
	\item [\itemref{item:comp-1}] for all 1-cells $f : a \onecell b$, $g : b \onecell c$, and displayed 1-cells $\ff : \dmor{\aa}{\bb}{f}$ and $\gg : \dmor{\bb}{\cc}{g}$, we have a displayed 1-cell $\ff \cdot \gg : \dmor{\aa}{\cc}{f \cdot g}$;
	\item [\itemref{item:id-2}] for all $f : \B_1(a,b)$, $\aa : \dob{\D}{a}$, $\bb : \dob{\D}{b}$, and $\ff : \dmor{\aa}{\bb}{f}$, we have $\did_2(\ff) : \dtwo{\ff}{\ff}{\id_2(f)}$;
	\item [\itemref{item:v-comp}] for 2-cells $\tc : f \mytwocell g$ and $\tcB : g \mytwocell h$, and displayed 2-cells $\dtc : \dtwo{\ff}{\gg}{\tc}$ and $\dtcB : \dmor{\gg}{\hh}{\tcB}$, we have a displayed 2-cell $\dtc \vcomp \dtcB : \dtwo{\ff}{\hh}{\tc \vcomp \tcB}$.
	\item [\itemref{item:left-whisker}] for each displayed 1-cell $\ff : \dmor{\aa}{\bb}{f}$ and each displayed 2-cell $\dtwo{\gg}{\hh}{\tc}$, we have a displayed 2-cell $\ff \whiskerl \dtc : \dtwo{\ff \cdot \gg}{\ff \cdot \hh}{f \whiskerl \tc}$ ;
	\item [\itemref{item:right_whisker}] for each displayed 1-cell $\hh : \dmor{\bb}{\cc}{h}$ and each displayed 2-cell $\dtc : \dtwo{\ff}{\gg}{\tc}$, we have a displayed 2-cell $\dtc \whiskerr \hh : \dtwo{\ff \cdot \hh}{\gg \cdot \hh}{\tc \whiskerr h}$;
	\item [\itemref{item:lunitor}] for each $\ff : \dmor{\aa}{\bb}{f}$, we have displayed 2-cells $\lunitor(\ff) : \dtwo{\id_1(\aa) \cdot \ff}{\ff}{\lunitor(f)}$ and $\lunitor(\ff)^{-1} : \dtwo{\ff}{\id_1(\aa) \cdot \ff}{\lunitor(f)^{-1}}$;
	\item [\itemref{item:runitor}] for each $\ff : \dmor{\aa}{\bb}{f}$, displayed 2-cells $\runitor(\ff) : \dtwo{\ff \cdot \id_1(\bb)}{\ff}{\runitor(f)}$ and $\runitor(\ff)^{-1} : \dtwo{\ff}{\ff \cdot \id_1(\bb)}{\runitor(f)^{-1}}$;
	\item [\itemref{item:lassociator}] for each $\ff : \dtwo{\aa}{\bb}{f}$, $\gg : \dtwo{\bb}{\cc}{g}$, and $\hh : \dtwo{\cc}{\dd}{h}$, we have displayed 2-cells $\assoc(\ff,\gg,\hh) : \dtwo{\ff \cdot (\gg \cdot \hh)}{(\ff \cdot \gg) \cdot \hh}{\assoc(f,g,h)}$ and $\assoc(\ff,\gg,\hh)^{-1} : \dtwo{(\ff \cdot \gg) \cdot \hh}{\ff \cdot (\gg \cdot \hh)}{\assoc(f,g,h)^{-1}}$.
\end{description}
 Note that we use the same notation for the displayed and the non-displayed operations.

 These operations are subject to laws, which are derived  systematically from the non-displayed version.
 Just as for displayed 1-categories, the laws of displayed bicategories are heterogeneous, because they are transported along the analogous law in the base bicategory.
 For instance, the displayed left-unitary law for identity reads as  $\id_2(\ff) \vcomp \dtc =_e \dtc$, 
 where $e$ is the corresponding identity of \Cref{item:vcomp-l-r-assoc} in \Cref{def:bicat}.
 \begin{description}
 	\item [\itemref{item:vcomp-l-r-assoc} ] $\id_2(f) \vcomp \tc =_* \tc, \quad \tc \vcomp \id_2(g) =_* \tc, \quad \tc \vcomp (\tcB \vcomp \tcC) =_* (\tc \vcomp \tcB) \vcomp \tcC$;
 	\item [\itemref{item:lwhisker-id-comp}] $f \whiskerl (\id_2 g) =_* \id_2(f \cdot g), \quad f \whiskerl (\tc \vcomp \tcB) =_* (f \whiskerl \tc) \vcomp (f \whiskerl \tcB)$;
 	\item [\itemref{item:rwhisker-id-comp}] $(\id_2 f) \whiskerr g =_* \id_2(f \cdot g), \quad (\tc \vcomp \tcB) \whiskerr g =_* (\tc \whiskerr g) \vcomp (\tcB \whiskerr g)$;
 	\item [\itemref{item:id-lwhisker-vcomp}] $(\id_1(a) \whiskerl \tc) \vcomp \lambda(g) =_* \lambda(f) \vcomp \tc$;
 	\item [\itemref{item:id-rwhisker-vcomp}] $(\tc \whiskerr \id_1(b)) \vcomp \rho(g) =_* \rho(f) \vcomp \tc$;
 	\item [\itemref{item:lwhisker-assoc}] $(f \whiskerl (g \whiskerl \tc)) \vcomp \alpha(f, g, i) =_* \alpha(f,g,h) \vcomp ((f \cdot g) \whiskerl \tc)$;
 	\item [\itemref{item:lwhisker-rwhisker-assoc}] $(f \whiskerl (\tc \whiskerr i)) \vcomp \alpha(f,h,i)  =_* \alpha(f,g,i) \vcomp ((f \whiskerl \tc) \whiskerr i)$;
 	\item [\itemref{item:rwhisker-assoc}] $(\tc \whiskerr (h \cdot i)) \vcomp \alpha(g,h,i) =_* \alpha(f,h,i) \vcomp ((\tc \whiskerr h) \whiskerr i)$;
 	\item [\itemref{item:vcomp-whisker}] $(\tc \whiskerr h) \vcomp (g \whiskerl \tcB) =_* (f \whiskerl \tcB) \vcomp (\tc \whiskerr i)$;
 	\item [\itemref{item:lambda}] $\lambda(f) \vcomp \lambda(f)^{-1} =_* \id_2(\id_1(a) \cdot f), \quad \lambda(f)^{-1} \vcomp \lambda(f) =_* \id_2(f)$;
 	\item [\itemref{item:rho}] $\rho(f) \vcomp \rho(f)^{-1} =_* \id_2(f \cdot \id_1(b)), \quad \rho(f)^{-1} \vcomp \rho(f) =_* \id_2(f)$;
 	\item [\itemref{item:alpha}] $\alpha(f,g,h) \vcomp \alpha(f,g,h)^{-1} =_* \id_2(f \cdot (g \cdot h)), \quad \alpha(f,g,h)^{-1} \vcomp \alpha(f,g,h) =_* \id_2((f \cdot g) \cdot h)$;
 	\item [\itemref{item:tri}] $\alpha(f, \id_1(b),g) \vcomp (\rho(f) \whiskerr g) =_* f \whiskerl \lambda(f)$;
 	\item [\itemref{item:pent}]$\alpha(f,g,h \cdot i) \vcomp \alpha(f \cdot g, h, i) =_* (f \whiskerl \alpha(g,h,i)) \vcomp \alpha(f,g \cdot h, i) \vcomp (\alpha(f,g,h) \whiskerr i)$.
 \end{description}

\end{defi}

The purpose of displayed bicategories is to give rise to a total bicategory together with a projection pseudofunctor.
They are defined as follows:

\begin{defi}[\coqident{Bicategories.DisplayedBicats.DispBicat}{total_bicat}]
\label{def:total-bicat}
  Given a displayed bicategory $\D$ over a bicategory $\B$, we form the \fat{total bicategory} $\total{\D}$ (or $\total[\B]{\D}$) which has:
  \begin{enumerate}
  \item as 0-cells tuples $(a, \aa)$, where $a : \B$ and $\aa : \dob{\D}{a}$;
  \item as 1-cells tuples $(f, \ff) : (a,\aa) \onecell (b,\bb)$, where $f : a \onecell b$ and $\ff : \dmor{\aa}{\bb}{f}$;
  \item as 2-cells tuples $(\tc, \dtc) : (f, \ff) \mytwocell (g, \gg)$, where $\tc : f \mytwocell g$ and $\dtc : \dtwo{\ff}{\gg}{\tc}$.
  \end{enumerate}
  We also have a \fat{projection pseudofunctor} $\dproj_D : \pseudoF{\total{D}}{\B}$.
\end{defi}

As mentioned before, the bicategory of pointed 1-types is the total bicategory of the following displayed bicategory.

\begin{example}[\coqident{Bicategories.DisplayedBicats.Examples.PointedOneTypes}{p1types_disp}, \Cref{ex:univalent}, \Cref{ex:one_types:univalent} cont'd]
  \label{ex:p1types_disp}
Given a universe $\U$, we build a displayed bicategory of pointed 1-types over the base bicategory of 1-types in $\U$ (\Cref{ex:one_types}).
\begin{itemize}
	\item For 1-type $A$ in $\U$, the objects over $A$ are inhabitants of $A$.
	\item For $f : A \function B$ with $A, B$ 1-types in $\U$, the displayed 1-cells over $f$ from $a$ to $b$ are paths $f(a) = b$.
	\item Given two functions $f, g : A \function B$, a homotopy $p : f \sim g$, two points $a : A$ and $b : B$, and paths $q_f : f(a) = b$ and $q_g : g(a) = b$,
	the 2-cells over $p$ are paths $q_f = p(a) \vcomp q_g$. 
\end{itemize}
The bicategory of pointed 1-types is the total bicategory of this displayed bicategory.
\end{example}

\begin{example}[\coqident{Bicategories.DisplayedBicats.Examples.PointedGroupoid}{pgrpds}]
	\label{ex:pgrpds_disp}
	We define a displayed bicategory of pointed groupoids over the base bicategory $\Grpd$ of groupoids.
	\begin{itemize}
		\item For a groupoid $G$, the objects over $G$ are objects of $G$.
		\item For a functor $F : G_1 \onecell G_2$ between groupoids $G_1$ and $G_2$, the displayed 1-cells over $F$ from $x$ to $y$ are isomorphisms $F(a) \iso b$.
		\item Given two functors $F_1, F_2 : G_1 \function G_2$, a natural transformation $n : F_1 \mytwocell F_2$, two points $x : G_1$ and $y : G_2$, and isomorphisms $q_1 : F_1(x) \iso y$ and $q_2 : F_2(x) = y$,
		the displayed 2-cells over $n$ are paths $p(a) \vcomp q_g = q_f$. 
	\end{itemize}
	The bicategory of pointed groupoids is the total bicategory of this displayed bicategory.
\end{example}

\begin{example}[\coqident{Bicategories.DisplayedBicats.Examples.FullSub}{disp_fullsubbicat}]
\label{ex:dispfullsub}
  Given a bicategory $\B$ and a predicate on 0-cells $\pred : \B_0 \to \hProp$, 
  define a displayed bicategory $\D$ over $\B$ such that $\dob{\D}{x} \eqdef \pred(x)$, and the types of displayed 1-cells and 2-cells are the unit type.
  The total bicategory of $\D$ provides a formal construction of the full subbicategory of $\B$ with cells satisfying $\pred$  introduced in \Cref{ex:fullsub}.
  In particular, a 1-cell in the total bicategory of $\D$ is a pair consisting of a 1-cell from $\B$ and the unique inhabitant of the unit type. Similarly for 2-cells.
\end{example}

We end this section presenting several general constructions of displayed bicategories.
\begin{defi}[Various constructions of displayed bicategories]
  \label{ex:disp_bicat}
  ~
  \begin{enumerate}
  \item (\coqident{Bicategories.DisplayedBicats.Examples.Prod}{disp_dirprod_bicat})
    Given displayed bicategories $\D_1$ and $\D_2$ over a bicategory $\B$,
    we construct the product $\D_1 \times \D_2$ over $\B$. The 0-cells, 1-cells, and 2-cells are pairs of 0-cells, 1-cells, and 2-cells respectively. \label{ex:disp_dirprod}
  \item (\coqident{Bicategories.DisplayedBicats.Examples.Sigma}{sigma_bicat})
    Given a displayed bicategory $\D$ over a base $\B$ and a displayed bicategory $\E$ over $\total{\D}$,
    we construct the \fat{sigma displayed bicategory} $\dsigma[\D]{\E}$ over $\B$ as follows.
    The objects over $a : \B$ are pairs $(\aa, e)$, where $\aa : \dob{\D}{a}$ and $e : \dob{\E}{(a,\aa)}$,
    the morphisms over $f : \mor{a}{b}$ from $(\aa, e)$ to $(\bb, e')$ are pairs $(\ff, \varphi)$, where $\ff : \dmor{\aa}{\bb}{f}$ and $\varphi : \dmor{e}{e'}{(f, \ff)}$, and similarly for 2-cells.\label{ex:disp_sigma}
  \item (\coqident{Bicategories.DisplayedBicats.Examples.Trivial}{trivial_displayed_bicat})
    Every bicategory $\D$ is, in a trivial way, a displayed bicategory
    over any other bicategory $\B$.
    Its total bicategory is the \fat{direct product} $\B \times \D$.
  \item (\coqident{Bicategories.DisplayedBicats.Examples.DisplayedCatToBicat}{disp_cell_unit_bicat})
    We say a displayed bicategory $\D$ over $\B$ is \fat{locally chaotic} if, for each $\alpha : f \mytwocell g$ and $\ff : \dmor{\aa}{\bb}{f}$ and $\gg : \dmor{\aa}{\bb}{g}$, the type $\dtwo{\ff}{\gg}{\alpha}$ is contractible.
    Let $\B$ be a bicategory and suppose we have 
    \begin{itemize}
    \item for each object $a$ in $\B$ a type $\dob{\D}{b}$ of displayed 0-cells;
    \item for each 1-cell $f : a \onecell b$ in $\B$ and for each $\aa : \dob{\D}{a}, \bb : \dob{\D}{b}$ a type $\dmor{\aa}{\bb}{f}$ of displayed 1-cells;
    \item displayed 1-identities $\id_1$ and compositions $(\cdot)$ of displayed 1-cells as in \Cref{def:disp_bicat}.
    \end{itemize}
    Then we have an associated
    \fat{locally chaotic displayed bicategory} $\hat{\D}$ over $\B$ by
    stipulating that the types of 2-cells are the unit type.
    Note that this construction essentially gives a way of obtaining a displayed bicategory from the data of a displayed category~\cite[Def.~3.1, Items~1--4]{AhrensL19}.
    \label{ex:chaotic_disp_bicat}
  \end{enumerate}
\end{defi}

Now let us discuss two more examples of bicategories obtained from displayed bicategories: firstly, monads internal to an arbitrary bicategory and secondly, Kleisli triples.
In \Cref{constr:monadbiequiv}, we construct a biequivalence between the bicategory of Kleisli triples and the bicategory of monads internal to $\Cat$.

\begin{defi}[\coqident{Bicategories.DisplayedBicats.Examples.Monads}{monad}]
\label{def:monads}
Let $\B$ be a bicategory.
Then we define a displayed bicategory $\M(\B)$ over $\B$ such that
\begin{itemize}
	\item The displayed objects over $a : \B$ are monad structures on $a$.
	A monad structure on $a$ consists of a 1-cell $m_a : a \onecell a$ and 2-cells $\eta_a : \id_1(a) \mytwocell m$ and $\mu_a : m \cdot m \mytwocell m$ such that the following diagrams commute
	\[
	\xymatrixcolsep{2.7pc}\xymatrix
	{
		f \cdot \id_1 \twoar[r]^-{f \whiskerl \eta} \twoar[dr]_{\lunitor(f)} &
		f \cdot f \twoar[d]_-{\mu} &
		\id_1 \cdot f \twoar[l]_-{\eta \whiskerr f} \twoar[dl]^{\runitor(f)} \\
		{} & f
	}
	\qquad
	\xymatrixcolsep{1.4pc}\xymatrix
	{
		f \cdot (f \cdot f) \twoar[d]_{f \whiskerl \mu} \twoar[rr]^-{\assoc(f,f,f)} &  & (f \cdot f) \cdot f \twoar[rr]^-{\mu \whiskerr f} & & f \cdot f \twoar[d]^{\mu}\\
		f \cdot f \twoar[rrrr]_-{\mu} & & & & f
	}
	\]
	\item The displayed 1-cells over $f : a \onecell b$ from $(m_a, \eta_a, \mu_a)$ to $(m_b, \eta_b, \mu_b)$ consist of invertible 2-cells $n_f : m_a \cdot f \mytwocell f \cdot m_b$
	such that the following two diagrams commute
	\[
	\xymatrix
	{
		\id_1(a) \cdot f \twoar[r]^-{\eta_a \whiskerr f} \twoar[d]_{\lunitor(f)} & m_a \cdot f \twoar[r]^-{n} & f \cdot m_b\\
		f \twoar[rr]_-{\runitor(f)^{-1}} & & f \cdot \id_1(b) \twoar[u]_-{f \whiskerr \eta_b}
	}
	\]
	\[
	\xymatrix
	{
		(m_a \cdot m_a) \cdot f \twoar[rr]^-{\mu_a \whiskerr f} \twoar[d]_-{\assoc(m_a,m_a,f)^{-1}} & & m_a \cdot f \twoar[rr]^-{n} & & f \cdot m_b\\
		m_a \cdot (m_a \cdot f) \twoar[d]_-{m_a \whiskerl n} & & & & f \cdot (m_b \cdot m_b) \twoar[u]_-{f \whiskerl \mu_b}\\
		m_a \cdot (f \cdot m_b) \twoar[rr]_-{\assoc(m_a,f,m_b)} & & (m_a \cdot f) \cdot m_b \twoar[rr]_-{n \whiskerr m_b} & & (f \cdot m_b) \cdot m_b \twoar[u]_-{\assoc(f,m_b,m_b)^{-1}}
	}
	\]
	\item The displayed 2-cells over $x : f \mytwocell g$ from $n_f$ to $n_g$ are proofs that the following diagrams commute
	\[
	\xymatrix
	{
		m_a \cdot f \twoar[r]^-{m_x \whiskerl x} \twoar[d]_{n_f} & m_a \cdot g \twoar[d]^{n_g} \\
		f \cdot m_b \twoar[r]_-{x \whiskerr m_x} & g \cdot m_b
	}
	\]
\end{itemize}
The total bicategory of $\M(\B)$ is the bicategory of \fat{monads internal to $\B$}.
\end{defi}

Next, we define a bicategory of Kleisli triples (also known as extension systems~\cite{marmolejo2010monads}).
\begin{defi}[\coqident{Bicategories.DisplayedBicats.Examples.KleisliTriple}{kleisli_triple_disp_bicat}]
\label{def:ktriple}
We define a displayed bicategory $\K$ over $\Cat$ such that
\begin{itemize}
	\item The displayed objects over $C$ are Kleisli triples over $C$. These consist of a function $M : C_0 \rightarrow C_0$, for each $a : C$ an arrow $\eta(a) : a \onecell M(a)$, and for each arrow $f : a \onecell M(b)$, an arrow $f^* : M(a) \onecell M(b)$ such that the usual laws hold.
	\item The displayed 1-cells over a functor $F : C \onecell D$ from $M_C$ to $M_D$ consists of isomorphisms $F_M$ from $M_D(F(a))$ to $F(M_C(a))$ for each $a : C_0$ such that the usual laws hold.
	\item The displayed 2-cells over $n : F \mytwocell G$ from $F_M$ to $G_M$ are equalities
	\[
	F_M(a) \cdot n(M_C(a)) = M_D(n(a)) \cdot G_M(a).
	\]
\end{itemize}
The total bicategory of $\K$ is the \fat{bicategory of Kleisli triples}.
\end{defi}

\section{Displayed Univalence}
\label{sec:disp_univalence}
Given a bicategory $\B$ and a displayed bicategory $\D$ over $\B$, our goal is to prove the univalence of $\total{\D}$ from conditions on $\B$ and $\D$.
For that, we develop the notion of \emph{univalent displayed bicategories}.
We start by defining displayed versions of invertible 2-cells.

\begin{defi}[\coqident{Bicategories.DisplayedBicats.DispBicat}{is_disp_invertible_2cell}]
Given are a bicategory $\B$ and a displayed bicategory $\D$ over $\B$.
Suppose we have objects $a, b : \B_0$, two 1-cells $f, g : \B_1(a,b)$, and an invertible 2-cell $\tc : \B_2(f,g)$.
Suppose that we also have $\aa : \dob{D}{a}$, $\bb : \dob{D}{b}$, $\ff : \dmor{\aa}{\bb}{f}$, $\gg : \dmor{\aa}{\bb}{g}$, and $\dtc : \dtwo{\ff}{\gg}{\tc}$.
Then we say $\dtc$ is \fat{invertible} if we have $\dtcB : \dtwo{\gg}{\ff}{\tc^{-1}}$ such that $\dtc \vcomp \dtcB$ and $\dtcB \vcomp \dtc$ are identities modulo transport over the corresponding identity laws of $\tc$.

A \fat{displayed invertible 2-cell over $\tc$}, where $\tc$ is an invertible 2-cell, is a pair of a displayed 2-cell $\dtc$ over $\tc$ and a proof that $\dtc$ is invertible.
The type of displayed invertible 2-cells from $\ff$ to $\gg$ over $\tc$ is denoted by $\diso{\ff}{\tc}{\gg}$.
\end{defi}
Being a displayed invertible 2-cell is a proposition and the displayed 2-cell $\id_2(\ff)$ over $\id_2(f)$ is invertible.
Next we define displayed adjoint equivalences.

\begin{defi}[\coqident{Bicategories.DisplayedBicats.DispAdjunctions}{disp_left_adjoint_equivalence}]
Given are a bicategory $\B$ and a displayed bicategory $\D$ over $\B$.
Suppose we have objects $a, b : \B_0$ and a 1-cell $f : \B_1(a,b)$ together with an adjoint equivalence structure $A$ on $f$.
We write $r$, $\eta$, $\varepsilon$ for the right adjoint, unit, and counit of $f$ respectively.
Furthermore, suppose that we have $\aa : \dob{\D}{a}$,$\bb : \dob{\D}{b}$, and $\ff : \dmor{\aa}{\bb}{f}$.
A \fat{displayed adjoint equivalence structure} on $\ff$ consists of
\begin{itemize}
	\item A displayed 1-cell $\rr : \dmor{\bb}{\aa}{r}$;
	\item An invertible displayed 2-cell $\dtwo{\did_1(\aa)}{\ff \cdot \rr}{\eta}$;
	\item An invertible displayed 2-cell $\dtwo{\rr \cdot \ff}{\did_1(\bb)}{\varepsilon}$.
\end{itemize}
In addition, two laws reminiscent of those in \Cref{def:adjequiv} need to be satisfied.

A \fat{displayed adjoint equivalence} over the adjoint equivalence $A$ is a pair of a displayed 1-cell $\ff$ over $f$ together with a displayed adjoint equivalence structure on $\ff$.
The type of displayed adjoint equivalences from $\aa$ to $\bb$ over $f$ is denoted by $\dadjequiv{\aa}{f}{\bb}$.
\end{defi}
The displayed 1-cell $\id_1(\aa)$ is a displayed adjoint equivalence over $\id_1(a)$.

Using these definitions, we define univalence of displayed bicategories similarly to univalence for ordinary bicategories.
Again we separate it in a local and global condition.

\begin{defi}[\coqfile{Bicategories.DisplayedBicats}{DispUnivalence}]
  Let $\D$ be a displayed bicategory over $\B$.
\begin{enumerate}
\item
  Let $a, b : \B$, and $\aa : \dob{\D}{a}, \bb : \dob{\D}{b}$.
  Let $f, g : \mor{a}{b}$, let $p : f = g$, and let $\ff$ and $\gg$ be displayed morphisms over $f$ and $g$ respectively.
  Then we define a function
  \[
  \dispidtoiso^{2,1}_{p,\ff,\gg} : \ff \depeq[p] \gg \function \diso{\ff}{\idtoiso^{2,1}_{f,g}(p)}{\gg}
  \]
  sending $\idpath$ to the identity displayed isomorphism.
  We say that $\D$ is \fat{locally univalent} if the function $\dispidtoiso^{2,1}_{p,\ff,\gg}$ is an equivalence for each $p$, $\ff$, and $\gg$.
\item
Let $a, b : \B$, and $\aa : \dob{\D}{a}, \bb : \dob{\D}{b}$.
Given $p : a = b$, we define a function
\[
\dispidtoiso^{2,0}_{p,\aa,\bb} : \aa \depeq[p] \bb \function \dadjequiv{\aa}{\idtoiso^{2,0}_{a,b}(p)}{\bb}
\]
sending $\idpath$ to the identity displayed adjoint equivalence.
We say that $\D$ is \fat{globally univalent} if the function $\dispidtoiso^{2,0}_{p,\aa,\bb}$ is an equivalence for each $p$, $\aa$, and $\bb$.
\item
  (\coqident{Bicategories.DisplayedBicats.DispUnivalence}{disp_univalent_2})
  We call $\D$ \fat{univalent} if it is both locally and globally univalent.
\end{enumerate}
\end{defi}

The following result states that univalence of the total bicategory can be proved from univalence of the base and of the displayed part.
This is the bicategorical version of the analogous result for 1-categories shown in~\cite[Theorem~7.4]{AhrensL19}, which in turn generalizes the Structure Identity Principle~\cite[Theorem 9.8.2]{hottbook}.

\begin{theorem}[\coqident{Bicategories.DisplayedBicats.DispUnivalence}{total_is_univalent_2}]
\label{th:univalence_total}
Let $\B$ be a bicategory and let $\D$ be a displayed bicategory over $\B$.  Then
\begin{enumerate}
\item
$\total{\D}$ is locally univalent if $\B$ is locally univalent and $\D$ is locally univalent;\label{th:univalence_total_local}
\item
$\total{\D}$ is globally univalent if $\B$ is globally univalent and $\D$ is globally univalent.\label{th:univalence_total_global}
\end{enumerate}
\end{theorem}

\begin{proof}
The main idea behind the proof is to characterize invertible 2-cells in the total bicategory as pairs of an invertible 2-cell $p$ in the base bicategory, and a displayed invertible 2-cell over $p$.
Concretely, for the local univalence of $\D$, we factor $\idtoiso^{2,1}$ as a composition of the following equivalences:
\[
\xymatrix@C=50pt
{
	(f , \ff) = (g , \gg) \ar[d]_{w_1}^{\sim}
	\ar[r]^-{\idtoiso^{2,1}}
	& \invtwocell\left((f,\ff),(g,\gg)\right)\\
	\Sum (p : f = g). \ff \depeq[p] \gg \ar[r]_-{w_2}^{\sim} & \Sum \left(p : \invtwocell(f,g)\right). \diso{\ff}{p}{\gg} \ar[u]_{w_3}^{\sim}
}
\]
The function $w_1$ is just a characterization of paths in a sigma type.
The function $w_2$ turns equalities into (displayed) invertible 2-cells, and it is an equivalence by local univalence of $\B$ and displayed local univalence of $\D$.
Finally, the function $w_3$ characterizes invertible 2-cells in the total bicategory.

The proof is similar in the case of global univalence.
The most important step is the characterization of adjoint equivalences in the total bicategory.
\[(a,\aa) \adjequiv (b, \bb) \xrightarrow{\sim}
\Sum (p : a \adjequiv b). \dadjequiv{\aa}{p}{\bb}. \qedhere \]
\end{proof}

To check displayed univalence, it suffices to prove the condition in the case where $p$ is reflexivity.
This step, done by path induction, simplifies some proofs of displayed univalence.

\begin{proposition}
  Given a displayed bicategory $\D$ over $\B$, then $\D$ is univalent if the following functions are equivalences:
  \begin{itemize}
  \item (\coqident{Bicategories.DisplayedBicats.DispUnivalence}{fiberwise_local_univalent_is_univalent_2_1})
    \[
      \dispidtoiso^{2,1}_{\idpath(f),\ff,\fff} : \ff = \fff \function \diso{\ff}{\id_2(f)}{\fff}
    \]
  \item (\coqident{Bicategories.DisplayedBicats.DispUnivalence}{fiberwise_univalent_2_0_to_disp_univalent_2_0})
    \[
      \dispidtoiso^{2,0}_{\idpath(a),\aa,\aaa} : \aa = \aaa \function \dadjequiv{\aa}{\id_1(a)}{\aaa}
    \]
  \end{itemize}
\end{proposition}
Now we establish the univalence of several examples.

\begin{example}
\label{ex:univalence}
  The following bicategories and displayed bicategories are univalent:
  \begin{enumerate}
  \item The category of pointed 1-types (see \Cref{ex:p1types_disp}) is univalent
    (\coqident{Bicategories.DisplayedBicats.Examples.PointedOneTypes}{p1types_univalent_2}).
  \item The full subbicategory (see \Cref{ex:fullsub}) of a univalent bicategory is univalent
    (\coqident{Bicategories.DisplayedBicats.Examples.FullSub}{is_univalent_2_fullsubbicat}).
  \item The product of univalent displayed bicategories (\Cref{ex:disp_bicat}, \Cref{ex:disp_dirprod}) is univalent
    (\coqident{Bicategories.DisplayedBicats.Examples.Prod}{is_univalent_2_dirprod_bicat}).
\end{enumerate}
\end{example}

For the sigma construction, we give two conditions for the univalence of the total bicategory.
If we have univalent displayed bicategories $\D_1$ and $\D_2$ over $\B$ and $\total{\D_1}$ respectively,
then we can either show the univalence of $\total{(\dsigma[\D_1]{\D_2})}$ directly
or we can show the displayed univalence of $\dsigma[\D_1]{\D_2}$.
Note that the second property could be necessary as an intermediate step for proving the univalence of a more complicated bicategory.
For the proof of displayed univalence of $\dsigma[\D_1]{\D_2}$, we need two assumptions on both displayed bicategories.

\begin{defi}[\coqident{Bicategories.DisplayedBicats.DispBicat}{disp_locally_groupoid}]
\label{def:locally_groupoid}
A displayed bicategory is \fat{locally groupoidal} if all its displayed 2-cells are invertible.
\end{defi}

\begin{defi}[\coqident{Bicategories.DisplayedBicats.DispBicat}{disp_2cells_isaprop}]
\label{def:locally_prop}
A displayed bicategory $\D$ over a bicategory $\B$ is called \fat{locally propositional} if the type $\dtwo{\ff}{\gg}{\tc}$ of displayed 2-cells over $\tc$ is a proposition.
\end{defi}

\begin{proposition}
\label{prop:disp_univ_sigma}
Let $\D_1$ and $\D_2$ be univalent displayed bicategories over univalent bicategories $\B$ and $\total \D_1$ respectively.
\begin{enumerate}
	\item The bicategory $\total{(\dsigma[\D_1]{\D_2})}$ (\Cref{ex:disp_bicat}, \Cref{ex:disp_sigma}) is univalent
	(\coqident{Bicategories.DisplayedBicats.Examples.Sigma}{sigma_is_univalent_2}).
	\item If $\D_1$ and $\D_2$ are locally propositional and groupoidal, then $\dsigma[\D_1]{\D_2}$ is displayed univalent (\coqident{Bicategories.DisplayedBicats.Examples.Sigma}{sigma_disp_univalent_2_with_props}) \label{item:univ_sigma}.
\end{enumerate}
\end{proposition}
We are not sure whether \Cref{item:univ_sigma} of \Cref{{prop:disp_univ_sigma}} is as strong as it can be---it might be possible to weaken the assumptions of $\D_1$ and $\D_2$ being locally propositional and groupoidal.
However, this would make the proof significantly more complicated.
In our examples these assumptions are satisfied, and thus the statement of \Cref{{prop:disp_univ_sigma}}, \Cref{item:univ_sigma} is sufficient for our purposes.

Lastly, we give a condition for when a locally chaotic displayed bicategory is univalent.
\begin{proposition}[\coqident{Bicategories.DisplayedBicats.Examples.DisplayedCatToBicat}{disp_cell_unit_bicat_univalent_2}]
  Let $\B$ be a univalent bicategory, and let $\D$ be a locally chaotic displayed bicategory (as in \Cref{ex:disp_bicat}, \Cref{ex:chaotic_disp_bicat}).
  Assume that for any $a : \B$, the type $\dob{\D}{a}$ is a set, and for any $\aa : \dob{\D}{a}, \bb : \dob{\D}{b}, f : a \onecell b$, the type $\dmor{\aa}{\bb}{f}$ is a proposition.
  Then $\D$ is univalent if we have a function in the opposite direction of $\dispidtoiso^{2,0}$.
\end{proposition}

\section{Displayed Constructions}
\label{sec:dispconstr}
The idea of building bicategories by layering displayed bicategories does not only allow for modular proofs of univalence,
but also for the modular construction of maps between bicategories, \eg pseudofunctors and biequivalences.
In this section, we introduce the notions of displayed pseudofunctor and biequivalence, and use them to build biequivalences.
The first example we look at, extends the biequivalence between 1-types and univalent groupoids in \Cref{ex:biequiv-grpds} to their pointed variants (\Cref{ex:p1types_disp} and \Cref{ex:pgrpds_disp}).

\begin{problem}
\label{prob:biequivpgrpds}
To construct a biequivalence between pointed 1-types and pointed groupoids.
\end{problem}

To construct the desired biequivalence, we first define \emph{displayed biequivalences} over a given biequivalence in the base
and we show that it gives rise to a total biequivalence on the total bicategories.
Since biequivalences are defined using pseudofunctors, pseudotransformations, and invertible modifications,
we first need to define displayed analogues of these.

\begin{defi}[\coqident{Bicategories.DisplayedBicats.DispPseudofunctor}{disp_psfunctor}]
Suppose we have bicategories $\B$ and $\C$, displayed bicategories $\D_1$ and $\D_2$ over $\B$ and $\C$ respectively, and a pseudofunctor $F : \pseudoF{\B}{\C}$.
Then a \fat{displayed pseudofunctor} $\FF$ from $\D_1$ to $\D_2$ over $F$ consists of
\begin{itemize}
	\item For each $a : \B$ a function $\FF_0 : \D_1(a) \rightarrow \D_2(F(a))$;
	\item For every 1-cell $f : a \onecell b$ and all displayed objects $\aa : \D_1(a)$ and $\bb : \D_1(b)$, a function sending 
	$f : \dmor{\aa}{\bb}{f}$ to $\FF_1(f) : \dmor{\FF_0(\aa)}{\FF_0(\bb)}{F(f)}$;
	\item For each 2-cell $\tc : f \mytwocell g$ and displayed 1-cells $\ff : \dmor{\aa}{\bb}{f}$ and $\gg : \dmor{\aa}{\bb}{g}$, a function
	sending $\thetatheta : \dtwo{\ff}{\gg}{\tc}$ to $\FF_2(\thetatheta) : \dtwo{\FF_1(\ff)}{\FF_1(\gg)}{F(\tc)}$;
	\item For all objects $a : \B$ and displayed objects $\aa : \D_1(a)$, we have a displayed invertible 2-cell
	$\identitor{\FF}(\xx) : \dtwo{\id_1(\FF_0(\xx))}{\FF_1(\id_1(\xx))}{\identitor{F}(x)}$;
	\item For all displayed 1-cells $\ff : \dmor{\aa}{\bb}{f}$ and $\gg : \dmor{\bb}{\cc}{g}$, we have a displayed invertible 2-cell
	$\compositor{\FF}(\ff,\gg) : \dtwo{\FF_1(\ff) \cdot \FF_1(\gg)}{\FF_1(\ff \cdot \gg)}{\compositor{F}(f,g)}$.
\end{itemize}
In addition, several laws similar to those in \Cref{def:psfun} need to hold.
They are just dependent variants of them and they hold over the corresponding non-dependent law.
Since the required laws are obtained in the same way as in \Cref{def:disp_bicat}, we do not show them here and instead refer the interested reader to the formalization.
We denote the type of displayed pseudofunctors from $\D_1$ to $\D_2$ over $F$ by $\disppsfun{\D_1}{\D_2}{F}$.
\end{defi}

\begin{defi}[\coqident{Bicategories.DisplayedBicats.DispTransformation}{disp_pstrans}]
Suppose that we have bicategories $\B$ and $\C$, pseudofunctors $F, G : \pseudoF{\B}{\C}$, and a pseudotransformation $\eta : \pstrans{F}{G}$.
Suppose furthermore that we have displayed bicategories $\D_1$ and $\D_2$ over $\B$ and $\C$, respectively, and displayed pseudofunctors $\FF$ and $\GG$ from $\D_1$ to $\D_2$ over $F$ and $G$, respectively.
Then a \fat{displayed pseudotransformation} $\etaeta$ over $\eta$ from $\FF$ to $\GG$ is given by
\begin{itemize}
	\item For each $x : \B$ and $\xx : \D_1(x)$ a displayed 1-cell $\etaeta_0(\xx) : \dmor{\FF_0(\xx)}{\GG_0(\xx)}{\eta_0(x)}$;
	\item For all 1-cells $f : x \onecell y$, displayed objects $\xx : \D_1(x)$ and $\yy : \D_1(y)$ and displayed 1-cells $\ff : \dmor{\xx}{\yy}{f}$,
	a displayed invertible 2-cell $\etaeta_1(\ff) : \dtwo{\etaeta_0(\xx) \cdot \FF_2(\ff)}{\FF_1(\ff) \cdot \etaeta_0(\yy)}{\eta_1(f)}$.
\end{itemize}
Again laws similar to those in \Cref{def:pstrans} need to hold and again they are derived similar to those in \Cref{def:disp_bicat}.
We denote the type of displayed pseudotransformations from $\FF$ to $\GG$ over $\eta$ by $\disppstrans{\FF}{\GG}{\eta}$.
\end{defi}

\begin{defi}[\coqident{Bicategories.DisplayedBicats.DispModification}{disp_modification}]
\label{def:disp_modification}
Suppose that we have bicategories $\B$ and $\C$, pseudofunctors $F, G : \pseudoF{\B}{\C}$, pseudotransformations $\eta, \theta : \pstrans{F}{G}$, and a modification $m : \modif{\eta}{\theta}$.
In addition, we are given displayed bicategories $\D_1$ and $\D_2$ over $\B$ and $\C$ respectively, displayed pseudofunctors $\FF : \disppsfun{\D_1}{\D_2}{F}$ and $\GG : \disppsfun{\D_1}{\D_2}{G}$, and displayed pseudotransformations $\etaeta : \disppstrans{\FF}{\GG}{\eta}$ and $\thetatheta : \disppstrans{\FF}{\GG}{\theta}$.
Then a \fat{displayed modification} from $\etaeta$ to $\thetatheta$ over $m$ is given by a displayed 2-cell $\dtwo{\etaeta_0(\xx)}{\thetatheta_0(\xx)}{m(x)}$ for each $x : \B$ and $\xx : \D_1(x)$.
In addition, the dependent version of the law in \Cref{def:modif} needs to hold.
We denote the type of displayed modifications from $\etaeta$ to $\thetatheta$ over $m$ by $\dispmodif{\etaeta}{\thetatheta}{m}$.  
\end{defi}
In order to formulate displayed biequivalence, we need an invertible version of \Cref{def:disp_modification}.
\begin{defi}[\coqident{Bicategories.DisplayedBicats.DispModification}{disp_invmodification}]
  A \fat{displayed invertible modification} over an invertible modification $m : \modif{\eta}{\theta}$ is a displayed modification
  $
    \mm : \dispmodif{\etaeta}{\thetatheta}{m}
  $
such that \[\mm(\xx) : \dtwo{\etaeta_0(\xx)}{\thetatheta_0(\xx)}{m(x)}\] is invertible for each $x : \B$ and $\xx : \D_1(x)$.
\end{defi}

Each of the discussed notions also has a total version.
These are constructed similarly to how the total bicategory is constructed in \Cref{def:total-bicat}.
\begin{problem}
\label{prob:totalgadget}
For each displayed gadget we discussed before, we have a total version.
\begin{itemize}
	\item (\coqident{Bicategories.DisplayedBicats.DispPseudofunctor}{total_psfunctor}) Given a displayed pseudofunctor $\FF : \disppsfun{\D_1}{\D_2}{F}$, to construct a pseudofunctor $\total{\FF} : \pseudoF{\total{\D_1}}{\total{\D_2}}$.
	\item (\coqident{Bicategories.DisplayedBicats.DispTransformation}{total_pstrans}) Given a displayed pseudotransformation $\etaeta : \disppstrans{\FF}{\GG}{\eta}$, to construct a pseudotransformation $\total{\etaeta} : \pstrans{\total{\FF}}{\total{\GG}}$.
	\item (\coqident{Bicategories.DisplayedBicats.DispModification}{total_invmodification}) Given a displayed invertible modification $\mm$ from $\etaeta$ to $\thetatheta$, to construct an invertible modification $\total{\mm} : \modif{\total{\etaeta}}{\total{\thetatheta}}$.
\end{itemize}
\end{problem}

\begin{construction}{prob:totalgadget}
\label{constr:totalgadget}
Each of the constructions is defined componentwise.
For example, $\total{\FF}$ on an object $(x, \xx)$ is defined to be $(F(x), \FF(\xx))$.
\end{construction}

To define displayed biequivalences, we need composition and identity of displayed pseudofunctors and pseudotransformations:

\begin{defi}
Suppose that $\B_1$, $\B_2$, and $\B_3$ are bicategories
and that $\D_1$, $\D_2$, and $\D_3$ are displayed bicategories over $\B_1$, $\B_2$, and $\B_3$, respectively.
In addition, let $F : \pseudoF{\B_1}{\B_2}$ and $G : \pseudoF{\B_2}{\B_3}$ be pseudofunctors
and suppose we have displayed pseudofunctors $\FF$ from $\D_1$ to $\D_2$ and $\GG$ from $\D_2$ to $\D_3$ over $F$ and $G$, respectively.
\begin{itemize}
	\item (\coqident{Bicategories.DisplayedBicats.DispPseudofunctor}{disp_pseudo_id})
	We have the \fat{identity displayed pseudofunctor} $\id(\D_1) : \disppsfun{\D_1}{\D_1}{\id(\B_1)}$.
	\item (\coqident{Bicategories.DisplayedBicats.DispPseudofunctor}{disp_pseudo_comp})
	We have a \fat{composition displayed pseudofunctor} $\FF \cdot \GG : \disppsfun{\D_1}{\D_3}{F \cdot G}$.
	\item (\coqident{Bicategories.DisplayedBicats.DispTransformation}{disp_id_pstrans})
	We have a \fat{displayed identity pseudotransformation} $\id_1(\FF) : \disppstrans{\FF}{\FF}{\id_1(F)}$.
	\item (\coqident{Bicategories.DisplayedBicats.DispTransformation}{disp_comp_pstrans})
	Suppose, we also have pseudofunctors $F', F'' : \pseudoF{\B_1}{\B_2}$ and pseudotransformations $\eta : \pstrans{F}{F'}$ and $\theta : \pstrans{F'}{F''}$.
	If we also have displayed pseudofunctors $\FFo : \disppsfun{\D_1}{\D_2}{F'}$ and $\FFt : \disppsfun{\D_1}{\D_2}{F''}$
	and displayed pseudotransformations $\etaeta : \disppstrans{\FF}{\FFo}{\eta}$ and $\thetatheta : \disppstrans{\FFo}{\FFt}{\theta}$,
	then we have a \fat{composition displayed pseudotransformation} $\etaeta \vcomp \thetatheta : \disppstrans{\FF}{\FFt}{\eta \vcomp \theta}$.
\end{itemize}
\end{defi}

Now we have developed sufficient displayed machinery to define displayed biequivalences.

\begin{defi}[\coqident{Bicategories.DisplayedBicats.DispBiequivalence}{disp_is_biequivalence_data}]
\label{example:disp_examples}
Let $\D_1$ and $\D_2$ be displayed bicategories over bicategories $\B$ and $\C$, respectively.
Suppose that we have a biequivalence $L : \pseudoF{\B}{\C}$.
We use the naming from \Cref{def:biequiv}.
Then a \fat{displayed biequivalence} from $\D_1$ to $\D_2$ over $L$ consists of
\begin{itemize}
	\item A displayed pseudofunctor $\LL : \disppsfun{\D_1}{\D_2}{L}$;
	\item A displayed pseudofunctor $\RR : \disppsfun{\D_2}{\D_1}{R}$;
	\item Displayed pseudotransformations $\etaeta : \disppstrans{\RR \cdot \LL}{\id(\D_2)}{\eta}$ and $\widedisp{\eta_i} : \disppstrans{\id(\D_2)}{\RR \cdot \LL}{\eta_i}$;
	\item Displayed pseudotransformation $\epseps : \disppstrans{\LL \cdot \RR}{\id(\D_1)}{\epsilon}$ and $\widedisp{\epsilon_i} : \disppstrans{\id(\D_1)}{\LL \cdot \RR}{\epsilon_i}$;
	\item Displayed invertible modifications
	\[\widedisp{m_1} : \dispmodif{\etaeta \vcomp \widedisp{\eta_i}}{\id_1(\RR \cdot \LL)}{m_1}
	\quad \quad
	\widedisp{m_2} : \dispmodif{\widedisp{\eta_i} \vcomp \etaeta}{\id_1(\id(\D_2))}{m_2}\]
	\item Displayed invertible modifications
	\[\widedisp{m_3} : \dispmodif{\epseps \vcomp \widedisp{\epsilon_i}}{\id_1(\LL \cdot \RR)}{m_3}
	\quad \quad
	\widedisp{m_4} : \dispmodif{\widedisp{\epsilon_i} \vcomp \epseps}{\id_1(\id(\D_1))}{m_4}\]
\end{itemize}
\end{defi}

Note that the total variant of each example in \Cref{example:disp_examples} is its non-displayed analogue.
Displayed biequivalences give rise to total biequivalences between their associated total bicategories.

\begin{problem}
\label{prob:totalbiequiv}
Let $\B$ and $\C$ be bicategories and suppose we have a biequivalence $L : \pseudoF{\B}{\C}$.
If we have displayed bicategories $\D_1$ and $\D_2$ over $\B$ and $\C$,
then each displayed biequivalence $\LL$ from $\D_1$ to $\D_2$ over $L$ gives rise to a biequivalence $\total{\LL}$ from $\total{\D_1}$ to $\total{\D_2}$.
\end{problem}

\begin{construction}[\coqident{Bicategories.DisplayedBicats.DispBiequivalence}{total_is_biequivalence}]{prob:totalbiequiv}
The pseudofunctors, pseudotransformations, and invertible modifications are constructed using \Cref{constr:totalgadget}.
\end{construction}

Note that to construct a displayed biequivalence, one must show several laws and construct multiple displayed invertible 2-cells.
If the involved displayed bicategories are locally groupoidal (\Cref{def:locally_groupoid}) and locally propositional (\Cref{def:locally_prop}),
then constructing a displayed biequivalence is simpler.
This is because all the necessary laws follow immediately from local propositionality and all the involved displayed 2-cells are invertible.
With all this in place, we finally show how to construct the desired biequivalence in \Cref{prob:biequivpgrpds} with displayed machinery.

\begin{construction}[\coqident{Bicategories.DisplayedBicats.Examples.PointedGroupoid}{disp_biequiv_data_unit_counit_path_pgroupoid}]{prob:biequivpgrpds}
By \Cref{prob:totalbiequiv} it suffices to construct a displayed biequivalence.
We only show how to construct the required displayed pseudofunctor from points on 1-types to points on groupoids.
\begin{itemize}
	\item Given a 1-type $X$ and a point $x : X$, we need to give an object of $\pathgroupoid(X)$, for which we take $x$.
	\item If we have 1-types $X$ and $Y$ with points $x : X$ and $y : Y$, and a function $f : X \onecell Y$ with a path $p_f : f(x) = y$,
	then we need to construct an isomorphism between $f(x)$ and $y$ in $\pathgroupoid(X)$.
	It is given by $p_f$.
	\item Suppose we have 1-types $X$ and $Y$ with points $x : X$ and $y : Y$.
	Furthermore, suppose we have a homotopy $s : f \sim g$ between functions $f, g : X \onecell Y$, paths $p_f : f(x) = y$ and $p_g : g(x) = y$,
	and a path $h : p_f = s(x) \vcomp p_g$.
	Then the required displayed 2-cell is the inverse of $h$.	
\end{itemize}
The compositor and the identitor are both the reflexivity path.
\end{construction}

As a final example, we construct a biequivalence between the bicategory of monads internal to $\Cat$ and the bicategory of Kleisli triples.

\begin{problem}
\label{prob:monadbiequiv}
To construct a biequivalence between monads and Kleisli triples.
\end{problem}

\begin{construction}[\coqident{Bicategories.DisplayedBicats.Examples.MonadKtripleBiequiv}{Monad_biequiv_Ktriple}]{prob:monadbiequiv}
\label{constr:monadbiequiv}
Note that the bicategory of monads and Kleisli triples are defined as the total bicategories of \Cref{def:monads} and \Cref{def:ktriple}, respectively.
Hence, by \Cref{prob:totalbiequiv}, it is sufficient to construct a displayed biequivalence between the respective displayed bicategories.
For the details on this construction, we refer the reader to the formalization.
\end{construction}

\section{Univalence of Complicated Bicategories}
\label{sec:examples}
In this section, we demonstrate the power of displayed bicategories on a number of complicated examples.
We show the univalence of the bicategory of pseudofunctors between univalent bicategories and of univalent categories with families.
In addition, we give two constructions to define univalent bicategories of algebras.

\subsection{Pseudofunctors}
\label{sec:pseudo}
As promised, we use displayed bicategories to prove \Cref{thm:psfunct_univalent_2}.
For the remainder, fix bicategories $\B$ and $\C$ such that $\C$ is univalent.
Recall that a pseudofunctor consists of an action on 0-cells, 1-cells, 2-cells, a family of 2-cells witnessing the preservation of composition and identity 1-cells, such that a number of laws are satisfied.

To construct the bicategory $\pseudo(\B,\C)$ of pseudofunctors, we start with a base bicategory whose objects are functions from $\B_0$ to $\C_0$.
Then we add structure to the base bicategory in several layers.
Each layer is given as a displayed bicategory over the total bicategory of the preceding one.
The first layer consists of actions of the pseudofunctors on 1-cells.
On its total bicategory, we define three displayed bicategories: one for the compositor, one for the identitor, and one for the action on 2-cells.
We take the total bicategory of the product of these three displayed bicategories.
Finally, we take the full subbicategory of that total bicategory on those objects that satisfy the axioms of a pseudofunctor.
To show its univalence, we show the base and each layer are univalent.

Now let us look at the formal definitions.
\begin{defi}[\coqident{Bicategories.PseudoFunctors.Display.Base}{ps_base}]
The bicategory $\Base(\B,\C)$ is defined as follows.
\begin{itemize}
	\item The objects are functions $\B_0 \function \C_0$;
	\item The 1-cells from $F_0$ to $G_0$ are families of 1-cells $\eta_0, \beta_0: \Prod (x : \B_0). F_0(x) \onecell G_0(x)$;
	\item The 2-cells from $\eta_0$ to $\beta_0$ are families of 2-cells $\modvar : \Prod (x : \B_0). \eta_0(x) \mytwocell \beta_0(x)$.
\end{itemize}
The operations are defined pointwise.
\end{defi}

Next we define a displayed bicategory over $\Base(\B,\C)$.
The displayed 0-cells are actions of pseudofunctors on 1-cells.
The displayed 1-cells over $\eta_0$ are 2-cells witnessing the naturality of $\eta_0$.
The displayed 2-cells over $\modvar$ are equalities which show that $\modvar$ is a modification.

\begin{defi}[\coqident{Bicategories.PseudoFunctors.Display.Map1Cells}{map1cells_disp_bicat}]
We define a displayed bicategory $\mapod(\B,\C)$ over $\Base(\B,\C)$ such that
\begin{itemize}
	\item the displayed objects over $F_0 : \B_0 \function \C_0$ are families of functions 
	\[F_1: \Prod (X, Y : \B_0). \B_1(X,Y) \function \C_1(F_0(X),F_0(Y));\]
	\item the displayed 1-cells over $\eta_0 : F_0(x) \onecell G_0(x)$ from $F_1$ to $G_1$ are families of invertible 2-cells
	\[
	\eta_1 : \Prod (X, Y : \B_0) (f : X \onecell Y). \eta_0(X) \cdot G_1(f) \mytwocell F_1(f) \cdot \eta_0(Y);
	\]
	\item the displayed 2-cells over $\modvar : \eta_0(x) \mytwocell \beta_0(x)$ from $\eta_1$ to $\beta_1$ are families of equalities
	\[
	\Prod (X, Y : \B_0) (f : X \onecell Y). \eta_1(f) \vcomp (F_1(f) \whiskerl \modvar(Y)) = (\modvar(X) \whiskerr G_1(f)) \vcomp \beta_1(f).
	\]
\end{itemize}
\end{defi}

We denote the total bicategory of $\mapod(\B,\C)$ by $\mapo(\B,\C)$.
Now we define three displayed bicategories over $\mapo(\B,\C)$.
Each of them is defined as a locally chaotic displayed bicategory (\Cref{ex:chaotic_disp_bicat} in \Cref{ex:disp_bicat}).

\begin{defi}[\coqident{Bicategories.PseudoFunctors.Display.Identitor}{identitor_disp_cat}]
We define a displayed bicategory $\mapid(\B,\C)$ over $\mapo(\B,\C)$ as follows:
\begin{itemize}
	\item The displayed objects over $(F_0,F_1)$ are identitors
	\[
	\identitor{F} : \Prod (X : \B_0). \id_1(F_0(X)) \mytwocell F_1(\id_1(X));
	\]
	\item The displayed morphisms over $(\eta_0, \eta_1)$ from $\identitor{F}$ to $\identitor{G}$ are equalities
	\[
	\runitor(\eta_0(X)) \vcomp \lunitor(\eta_0(X))^{-1} \vcomp (\identitor{F}(X) \whiskerr \eta_0(X)) = (\eta_0(X) \whiskerl \identitor{G}(X)) \vcomp \eta_1(\id_1(X)).
	\]
\end{itemize}
\end{defi}

\begin{defi}[\coqident{Bicategories.PseudoFunctors.Display.Compositor}{compositor_disp_cat}]
We define a displayed bicategory $\mapcd(\B,\C)$ over $\mapo(\B,\C)$ as follows:
\begin{itemize}
	\item The displayed objects over $(F_0,F_1)$ are compositors
	\[
	\compositor{F} : \Prod (X, Y, Z : \B_0) (f : \B_1(X,Y)) (g : \B_1(Y,Z)). F_1(f) \cdot F_1(g) \mytwocell F_1(f \cdot g);
	\]
	\item The displayed morphisms over $(\eta_0, \eta_1)$ from $\compositor{F}$ to $\compositor{G}$ consists of equalities
	\[\assoc \vcomp (\eta_1(f) \whiskerr G_1(g)) \vcomp \assoc^{-1} \vcomp (F_1(f) \whiskerl \eta_1(g)) \vcomp \assoc \vcomp (\compositor{F} \whiskerr \eta_0(Z))
	= (\eta_0(X) \whiskerl \compositor{G}) \vcomp \eta_1(f \cdot g)
	\]
	for all $X, Y, Z : \B_0$, $f : \B_1(X,Y)$ and $g : \B_1(Y,Z)$.
\end{itemize}
\end{defi}

\begin{defi}[\coqident{Bicategories.PseudoFunctors.Display.Map2Cells}{map2cells_disp_cat}]
We define a displayed bicategory $\maptd(\B,\C)$ over $\mapo(\B,\C)$ as follows:
\begin{itemize}
	\item The displayed objects over $(F_0, F_1)$ are
	\[
	F_2 : \Prod (a, b : \B_0) (f, g : a \onecell b). (f \mytwocell g) \function (F_1(f) \mytwocell F_1(g));
	\]
	\item The displayed morphisms over $(\eta_0, \eta_1)$ from $F_2$ to $G_2$ consist of equalities
	\[
	\Prod (\tc : f \mytwocell g).
	(\eta_0(X) \whiskerl G_2(\tc)) \vcomp \eta_1(g)
	=
	\eta_1(f) \vcomp (F_2(\tc) \whiskerr \eta_0(Y)).
	\]
\end{itemize}
\end{defi}

We denote the total category of the product of $\maptd(\B,\C)$, $\mapid(\B,\C)$, and $\mapcd(\B,\C)$ by $\rpseudo(\B,\C)$.
Note that its objects are of the form $((F_0, F_1), (F_2, \identitor{F}, \compositor{F}))$, its 1-cells are pseudotransformations, and its 2-cells are modifications.
However, its objects are not yet pseudofunctors, because those also need to satisfy the laws in \Cref{def:psfun}.

\begin{defi}[\coqident{Bicategories.PseudoFunctors.Display.PseudoFunctorBicat}{psfunctor_bicat}]
\label{def:psfunctor}
We define the bicategory $\pseudo(\B,\C)$ as the full subbicategory of $\rpseudo(\B,\C)$ where the objects satisfy the following laws
\begin{itemize}
	\item $F_2(\id_2(f)) = \id_2(F_1(f))$ and $F_2(f \vcomp g) = F_2(f) \vcomp F_2(g)$;
	\item $\lunitor(F_1(f)) = (\identitor{F}(a) \whiskerr F_1(f)) \vcomp \compositor{F}(\id_1(a), f) \vcomp F_2(\lunitor(f))$;
	\item $\runitor(F_1(f)) = (F_1(f) \whiskerl \identitor{F}(b)) \vcomp \compositor{F}(f, \id_1(b)) \vcomp F_2(\runitor(f))$;
	\item 
	$(F_1(f) \vcomp \compositor{F}(g,h)) \vcomp \compositor{F}(f, g \cdot h) \vcomp F_2(\assoc) = \assoc \vcomp (\compositor{F}(f, g) \whiskerr F_1(h)) \vcomp \compositor{F}(f \cdot g, h)$;
	\item $\compositor{F}(f, g_1) \vcomp F_2(f \whiskerl \tc) = (F_1(f) \whiskerl F_2(\tc)) \vcomp \compositor{F}(f, g_2)$;
	\item $\compositor{F}(f_1,g) \vcomp F_2(\tc \whiskerr g) = (F_2(\tc) \whiskerr F_1(g)) \vcomp \compositor{F}(f_2,g)$;
	\item $\identitor{F}(X)$ and $\compositor{F}(f,g)$ are invertible 2-cells.
\end{itemize}
\end{defi}

Note that the objects, 1-cells, and 2-cells of the resulting bicategory correspond to pseudofunctors (\Cref{def:psfun}), pseudotransformations (\Cref{def:pstrans}), and modifications (\Cref{def:modif}) respectively.
Each displayed layer in this construction is univalent.
In addition, if $\C$ is univalent, then so is $\Base(\B,\C)$.
All in all, the results of this subsection can be summarized as follows.
\begin{defi}
Given bicategories $\B$ and $\C$, we define a bicategory $\pseudo(\B,\C)$ whose objects are pseudofunctors, 1-cells are pseudotransformations, and 2-cells are modifications.
\end{defi}
\begin{theorem}
If $\C$ is univalent, then so is $\pseudo(\B,\C)$.
\end{theorem}

\subsection{Algebraic Examples}\label{sec:algebraic-examples} 
Next, we show how to use displayed bicategories to construct univalent bicategories of algebras for some signature.
We consider signatures that specify operations, equations, and coherencies on those equations.
More specifically, a signature consists of a pseudofunctor $F$ (specifying the operations),
a finite set of pairs of pseudotransformations $l_i$ and $r_i$ (specifying the equations),
and a proposition $P$ (specifying the coherencies) which can refer to $F$ and the $l_i$ and $r_i$.
An algebra on such a signature consists of an object $X$,
a 1-cell $h : F(X) \onecell X$,
2-cells $l_i(X) \mytwocell r_i(X)$,
such that the predicate $P$ is satisfied by all this data.

To define the bicategory of algebras on a signature, we define three displayed bicategories which add the operations, equations, and coherencies.
Since the equations can make use of the operations and the coherencies can refer to the equations,
the displayed bicategories must be layered suitably.
More specifically, starting with a bicategory $\B$ and a pseudofunctor $F : \pseudoF{\B}{\B}$, we first define a displayed bicategory whose displayed objects are algebras on $F$.
On top of its total bicategory, we give a displayed bicategory which adds 2-cells (modeling equations) to the structure.
This gives rise to another total bicategory. Finally, we consider the full subbicategory of the latter total bicategory consisting of all objects satisfying the desired coherencies.
The objects of the resulting total bicategory are models for the signature we started with.

To illustrate our approach, we show how to define the bicategory of monads internal to a bicategory, as discussed in \Cref{def:monads}.
A monad internal to a bicategory $\B$ consists of, among others, a 0-cell $X : \B$ and 1-cell $X \onecell X$ as an ``operation''.
Such structure is encapsulated by \emph{algebras for a pseudofunctor} and pseudomorphisms between those algebras.

\begin{defi}[\coqident{Bicategories.DisplayedBicats.Examples.Algebras}{disp_alg_bicat}]
\label{def:alg_bicat}
Let $\B$ be a bicategory and let $F : \pseudoF{\B}{\B}$ be a pseudofunctor.
We define a displayed bicategory $\algd(F)$.
\begin{itemize}
	\item The objects over $a: \B$ are 1-cells $F_0(a) \onecell a$.
	\item The 1-cells over $f : \B_1(a,b)$ from $h_a : F_0(a) \onecell a$ to $h_b : F_0(b) \onecell b$ are invertible 2-cells $h_a \cdot f \mytwocell F_1(f) \cdot h_b$.
	\item Given $f, g : \B_1(a,b)$, algebras $h_a : F_0(a) \onecell a$ and $h_b : F_0(b) \onecell b$, and $h_f$ and $h_g$ over $f$ and $g$ respectively, a 2-cell over $\tc : f \mytwocell g$ is a commuting square
	\[
	\xymatrix
	{
		h_a \cdot f \twoar[r]^-{h_f} \twoar[d]_{h_a \whiskerl \theta} & F_1(f) \cdot h_b \twoar[d]^{F_2(\theta) \whiskerr h_b} \\
		h_a \cdot g \twoar[r]_-{h_g} & F_1(g) \cdot h_b
	}
	\]
\end{itemize}
We write $\alg(F)$ for the total category of $\algd(F)$.
\end{defi}

\begin{theorem}[\coqident{Bicategories.DisplayedBicats.Examples.Algebras}{bicat_algebra_is_univalent_2}]
  \label{th:alg_univalence}
  Let $\B$ be a bicategory and let $F : \pseudoF{\B}{\B}$ be a pseudofunctor.
  If $\B$ is univalent, then so is $\alg(F)$.
\end{theorem}

\begin{example}[\Cref{ex:p1types_disp} cont'd]
  The bicategory of pointed 1-types is the bicategory of algebras for the constant pseudofunctor $F(a) = 1$.
\end{example}

Returning to the example of monads, define $\PM$ to be $\alg(\id(\B))$.
Objects of $\PM$ consist of an $X : \B_0$ and a 1-cell $X \onecell X$.
To refine this further, we need to add 2-cells corresponding to the unit and the multiplication.
We do this by defining two displayed bicategories over $\PM$.

In general, the construction for building algebras with 2-cells (which model ``equations'') looks as follows.
Suppose that we have a displayed bicategory $\D$ over some $\B$.
Our goal is to define a displayed bicategory over $\total{D}$ where the displayed 0-cells are certain 2-cells in $\B$.
The endpoints for these 2-cells are choices of 1-cells that are natural in objects, thus they are given by pseudotransformations $l, r$.
The source of the endpoints is $\dproj_D \cdot S$ for some $S : \pseudoF{\B}{\B}$, and the target is $\dproj_D \cdot \id(\B)$ where $\dproj_D$ is the projection from $\total{D}$ to $\B$.
The source pseudofunctor $S : \pseudoF{\B}{\B}$ determines the shape of the free variables that occur in the endpoints.
Note that the target of the endpoint is $\dproj_D \cdot \id(\B)$, instead of $\dproj_D$, which is symmetric to the source $\dproj_D \cdot S$.
This allows us to construct such transformations by composing them.

Thus, pseudotransformations  $l, r: \dproj_D \cdot S \onecell \dproj_D \cdot \id(\B)$ give 1-cells $l(a, h_a), r(a, h_a) : \B_1(S(a), a)$ for each $(a, h_a) : \total{D}$.
By allowing $l$ and $r$ to depend not only on the 0-cell $a : \B$, but also on the displayed cell $h_a : \D(a)$, the endpoints can refer to the operations that were added as part of algebras in \Cref{def:alg_bicat}.
Formally, the construction that adds 2-cells from $l(a)$ to $r(a)$ is defined as the following displayed bicategory.
\begin{defi}[\coqident{Bicategories.DisplayedBicats.Examples.Add2Cell}{add_cell_disp_cat}]
\label{def:add2cell}
Suppose that $\D$ is a displayed bicategory over $\B$.
Let $S : \pseudoF{\B}{\B}$ be a pseudofunctor and let $l, r :  \dproj_D \cdot S \onecell \dproj_D \cdot \id(\B)$ be pseudotransformations.
We define a displayed bicategory $\addcell(\D,l,r)$ over $\total{D}$ as a locally chaotic displayed bicategory (\cf \Cref{ex:chaotic_disp_bicat} in \Cref{ex:disp_bicat}).
\begin{itemize}
	\item The objects over $(a, h_a)$ are 2-cells $l(a, h_a) \mytwocell r(a, h_a)$.
	\item The morphisms over $(f, \ff) : \total{D}((a, h_a),(b, h_b))$ from $\alpha : l(a, h_a) \onecell r(a, h_a)$ to $\beta : l(b, h_b) \onecell r(b, h_b)$ are the following commuting squares in $\B$:
	\[
	\xymatrix
	{
		l(a, h_a) \cdot f \twoar[rrr]^-{\alpha \whiskerr f} \twoar[d]_{l(f, \ff)} & & & r(a, h_a) \cdot f \twoar[d]^{r(f, \ff)} \\
		S(f) \cdot l(b, h_b) \twoar[rrr]_-{S(f) \whiskerl \beta} & & & S(f) \cdot r(b, h_b)
	}
	\]
\end{itemize}
\end{defi}

\begin{theorem}
\label{th:add2cell_univalence}
The displayed bicategory $\addcell(\D,l,r)$ is locally univalent (\coqident{Bicategories.DisplayedBicats.Examples.Add2Cell}{add_cell_disp_cat_univalent_2_1}).
Moreover, if $\C$ is locally univalent and $\D$ is locally univalent, then $\addcell(\D,l,r)$ is globally univalent (\coqident{Bicategories.DisplayedBicats.Examples.Add2Cell}{add_cell_disp_cat_univalent_2_0}).
\end{theorem}

Returning to the example of monads, let us use \Cref{def:add2cell} to add the unit and the multiplication 2-cells to the structure of $\PM$.
We can add the unit and the multiplication separately, as two displayed bicategories.
For the unit, we pick the source pseudofunctor $S(a) = a$ and the endpoints are defined as $l(a, f : a \onecell a) = \id_0(a)$ and $r(a, f : a \onecell a) = f$.
For the multiplication, we use the same source pseudofunctor and the same right endpoint, but we pick the left endpoint to be $l(a, f : a \onecell a) = f \cdot f$.

Let $\LM'$ be the product of these two displayed bicategories, displayed over $\total{\PM}$.
We use the sigma construction (\cf \Cref{ex:disp_sigma} in \Cref{ex:disp_bicat}) to obtain a displayed bicategory $\LM$ over $\B$.
It is almost the bicategory of monads internal to $\B$.
To finalize the construction, we need to require the structures in $\LM$ to satisfy the monadic laws: for each object $(f, \eta, \mu)$ in $\total{\LM}$ the diagrams from \Cref{def:monads} need to commute.
We construct the final bicategory $\M(\B)$ (as in \Cref{def:monads}) as the full subbicategory of $\total{\LM}$ with respect to these laws.
Again to guarantee that $\M(\B)$ is displayed over $\B$, we use the sigma construction.
From \Cref{prop:disp_univ_sigma,th:alg_univalence,th:add2cell_univalence,ex:univalence} we conclude:
\begin{theorem}[\coqident{Bicategories.DisplayedBicats.Examples.Monads}{bigmonad_is_univalent_2}]
If $\B$ is univalent, then so is $\M(\B)$.
\end{theorem}

\subsection{Categories with Families}
\label{sec:cwfs}
Finally, we discuss the last example: the bicategory of (univalent) categories with families (CwFs) \cite{DBLP:conf/types/Dybjer95}.
We follow the formulation by Fiore (described as ``dependent context structures'' in \cite{fiore_cwf}) and Awodey \cite[Section 1]{awodey2018natural}, which is already formalized in \UniMath \cite{voevodsky2018categorical}:
a CwF consists of a category $C$, two presheaves $\Ty$ and $\Tm$ on $C$, a morphism $p : \Tm \to \Ty$,
and a representation structure for $p$.

However, rather than defining CwFs in one step, we use a stratified construction yielding
the sought bicategory as the total bicategory of iterated displayed layers.
The base bicategory is \Cat (cf.\ \Cref{ex:cat}).
The second layer of data consists of two presheaves, each described by the following construction.

\begin{defi}[\coqident{Bicategories.DisplayedBicats.Examples.ContravariantFunctor}{disp_presheaf_bicat}]
\label{def:disp_presheaf_bicat}
Define the displayed bicategory $\PShD$ over $\Cat$:
\begin{itemize}
	\item The objects over $C$ are functors from $\op{C}$ to the univalent category $\Set$;
	\item The 1-cells from $T : \op{C} \onecell \Set$ to $T' : \op{D} \onecell \Set$ over $F : C \onecell D$ are natural transformations from $T$ to $\op{F} \cdot T'$;
	\item The 2-cells from $\beta : T \mytwocell \op{F} \cdot T'$ to $\beta' : T \mytwocell \op{G} \cdot T'$ over $\gamma : F \mytwocell G$ are equalities
	\[
	\beta = \beta' \vcomp (\op{\gamma} \whiskerr T').
	\]
\end{itemize}
\end{defi}
Denote by $\TT$ the total category of the product of $\PShD$ with itself.
An object in $\TT$ consists of a category $C$ and two presheaves $\Ty, \Tm : \op{C} \onecell \Set$.

The next piece of data in a CwF is a natural transformation from $\Tm$ to $\Ty$:

\begin{defi}[\coqident{Bicategories.DisplayedBicats.Examples.Cofunctormap}{morphisms_of_presheaves_display}]
\label{def:cofunctormap}
We define a displayed bicategory $\CwFd$ on $\TT$ as the locally chaotic displayed bicategory (\Cref{ex:chaotic_disp_bicat} in \Cref{ex:disp_bicat}) such that 
\begin{itemize}
	\item The objects over $(C, (\Ty, \Tm))$ are natural transformations from $\Tm$ to $\Ty$.
	\item Suppose we have two objects $(C, (\Ty, \Tm))$ and $(C', (\Ty', \Tm'))$, two natural transformations $\p : \Tm \mytwocell \Ty$ and $\p' : \Tm' \mytwocell \Ty'$, and suppose we have a 1-cell $f$ from $(C, (\Ty, \Tm))$ to $(C', (\Ty', \Tm'))$. Note that $f$ consists of a functor $F : C \rightarrow C'$ and two transformations $\beta : \Ty \mytwocell \op{F} \circ \Ty'$ and $\beta' : \Tm \mytwocell \op{F} \circ \Tm'$.
	Then a 1-cell over $f$ is an equality
	\[
	p \vcomp \beta = \beta' \vcomp (\op{F} \whiskerl p').
	\]
\end{itemize}
\end{defi}
With $\CwFd$ and the sigma construction from \Cref{ex:disp_sigma} in \Cref{ex:disp_bicat}, we get a displayed bicategory over $\Cat$ and we denote its total bicategory by $\CwFh$.
As the last piece of data, we add the representation structure for the morphism $\p$ of presheaves.

\begin{defi}[\coqident{Bicategories.DisplayedBicats.Examples.CwF}{cwf_representation}]
Given a category $C$ together with functors $\Ty, \Tm : \op{C} \onecell \Set$ and a natural transformation $p : \Tm \mytwocell \Ty$, we say $\isCwF(C,\Ty,\Tm,p)$ if for each $\Gamma : C$ and $A : \Ty(\Gamma)$, we have a representation of the fiber of $\p$ over $A$.
\end{defi}

A detailed definition can be found in \cite[Definition 3.1]{voevodsky2018categorical}.
Since $C$ is univalent, the type $\isCwF(C,\Ty,\Tm,\p)$ is a proposition, and thus we define $\CwF$ as a full subbicategory of $\CwFh$.

\begin{proposition}[{\cite[Lemma~4.3]{voevodsky2018categorical} }, \coqident{Bicategories.DisplayedBicats.Examples.CwF}{isaprop_cwf_representation}]
$\isCwF(C,\Ty,\Tm,\p)$ is a proposition.
\end{proposition}

\begin{defi}[\coqident{Bicategories.DisplayedBicats.Examples.CwF}{cwf}]
\label{def:cwf}
We define $\CwF$ as the full subbicategory of $\CwFd$ with $\isCwF$.
\end{defi}

\begin{theorem}[\coqident{Bicategories.DisplayedBicats.Examples.CwF}{cwf_is_univalent_2}]
$\CwF$ is univalent.
\end{theorem}

\section{Displayed (2-)Inserters}
\label{sec:disp-inserters}
In this section, we study two general constructions which have been suggested by an anonymous referee.
Both the constructions and their name were suggested by the referee.
We already saw instances of them, namely in \Cref{sec:algebraic-examples,sec:cwfs}.

The first one, called the \emph{displayed inserter}, constructs a displayed bicategory whose total bicategory represents the inserter of two pseudofunctors.
A similar construction, namely inserters of 1-cells in bicategories, has already been studied in the literature.
Lambek defined subequalizers of functors \cite{lambek1970subequalizers}, and these are inserters in the bicategory of categories.
These inserters are also known as \emph{dialgebras}, and they have been used to study the semantics of inductive-inductive types \cite{AltenkirchMFS11}.
Power and Robison defined PIE-limits (products, inserters, equaifiers) in 2-categories and showed that they can be used to construct a general class of limits \cite{power1991characterization}.
In addition, it has been shown that bicategories of algebras are closed under inserters  \cite{blackwell1989two,VvdW}.
Note that the terminology ``displayed inserter'' has also been used for the inserter of displayed functors \cite{bocquet2021induction}, which is different from what we look at.

\begin{definition}[Displayed inserter, \coqident{Bicategories.DisplayedBicats.Examples.DisplayedInserter}{disp_inserter_bicat}]
 \label{def:disp_inserter}
 Let $\B$ and $\D$ be bicategories and let $F,G : \B \to \D$ be pseudofunctors.
 We define the following displayed bicategory over $B$, called the \fat{displayed inserter}:
 \begin{enumerate}
  \item displayed objects over $b : \B_0$ are 1-cells $g : \D_1(Fb,Gb)$;
  \item displayed 1-cells over $f:B_1(b,b')$ from $g : \D_1(Fb,Gb)$ to $g' : \D_1(Fb',Gb')$ are displayed 2-cells $\gamma$ as in
    \[
      \begin{xy}
        \xymatrix{
        Fb \ar[r]^{Ff} \ar[d]_{g}
        &
        Fb' \ar[d]^{g'}
        \\
        Gb \ar[r]_{Gf}
        &
        Gb' \ultwocell<\omit>{\gamma}
        }
      \end{xy}
    \]

  \item displayed 2-cells over \xymatrix{
         b \rtwocell^f_g{\theta}& b'
      }
      from $\gamma : g \cdot Gf \mytwocell Ff \cdot g'$ to $\gamma' : g \cdot Gf' \mytwocell Ff' \cdot g'$ are identities
      $\gamma \vcomp (F\theta \whiskerr g') = (g \whiskerl G\theta) \vcomp \gamma'$.

  \item composition of 1-cells is defined using whiskering and the associator in $\D$.
 \end{enumerate}
 The remaining properties are readily shown; we refer to the formalization for details.
\end{definition}

\begin{example}
\label{ex:1-inserters}
\Cref{def:alg_bicat,def:disp_presheaf_bicat} are---almost---instances of Definition~\ref{def:disp_inserter}.
Specifically, \Cref{def:alg_bicat} is obtained as the displayed inserter with $F$ the identity pseudofunctor and with $G$ the pseudofunctor $F$ of~\Cref{def:alg_bicat}.
However, this does not yet give the correct displayed 1-cells; we furthermore need to take the full sub-bicategory of \emph{invertible} displayed 1-cells (cf.\ \coqident{Bicategories.DisplayedBicats.Examples.Sub1Cell}{disp_sub1cell_bicat}). \Cref{def:disp_presheaf_bicat} is obtained by taking $F$ to be the identity on $\Cat$ and $G$ to be the functor that is constantly $\op{\Set}$.

Note that this is slightly different than in \Cref{def:disp_presheaf_bicat}, corresponding to the two ways to represent a contravariant functor $H : A \to B$ in terms of a covariant one---as a functor $H : \op{A} \to B$ or a functor $H : A \to \op{B}$.
While \Cref{def:disp_presheaf_bicat} uses the former, this is not possible here: domain and codomain of the inserter are specified by pseudofunctors, but the function $\op{(\_)} : \Cat_0 \to \Cat_0$ on categories does not extend to a pseudofunctor $\Cat \to \Cat$ that could take the place of the pseudofunctor $F$ above.
Instead, here we have to represent contravariant functors by taking the opposite of the target category, and thus consider the constant pseudofunctor returning the category $\op{\Set}$.
\end{example}

\begin{proposition}[\coqident{Bicategories.DisplayedBicats.Examples.DisplayedInserter}{disp_inserter_bicat_univalent_2_0}]
 Suppose given data as in \Cref{def:disp_inserter}. Then the displayed inserter is
 \begin{enumerate}
  \item locally univalent;
  \item globally univalent if $\B$ and $\C$ are locally univalent.
 \end{enumerate}
\end{proposition}

Next we look at \emph{displayed 2-inserters}.
These are quite similar to displayed inserter, but with one main difference: instead 1-cells, 2-cells are added to the structure.
More precisely, given two pseudotransformations $\alpha$ and $\beta$, the displayed 2-inserter gives a displayed bicategory of maps from $\alpha(x)$ to $\beta(x)$ for every $x$.

\begin{definition}[Displayed 2-inserter, \coqident{Bicategories.DisplayedBicats.Examples.Displayed2Inserter}{disp_two_inserter_bicat}]
  \label{def:disp-2-inserter}
  Let $\B$ and $\D$ be bicategories, let $F,G : \B \to \D$ be pseudofunctors, and let $\alpha,\beta : F \mytwocell G$ be pseudotransformations.
  We define the following locally chaotic displayed bicategory over $\B$, called the \fat{displayed 2-inserter}:
  \begin{enumerate}
  \item displayed objects over $b : \B_0$ are 2-cells $r$ in $\D$ as in
   \[
       \xymatrix@C=4em@R=4em{
         F_0b \rtwocell^{\alpha(b)}_{\beta(b)}{r}& G_0b
        }
  \]
  \item displayed 1-cells over $f:B_1(b,b')$ from
        \xymatrix@C=4em@R=4em{
         F_0b \rtwocell^{\alpha(b)}_{\beta(b)}{r}& G_0b
        }
        to
        \xymatrix@C=4em@R=4em{
         F_0b' \rtwocell^{\alpha(b')}_{\beta(b')}{s}& G_0b'
        }
        are identities
        \[ (r \whiskerr G_1 f) \vcomp \beta(f) = \alpha(f) \vcomp (F_1 f \whiskerl s) . \]
  \end{enumerate}
\end{definition}
\begin{example}
 The displayed bicategory of \Cref{def:add2cell} is immediately a displayed 2-inserter.

 The displayed bicategory of \Cref{def:cofunctormap} can be obtained as the following displayed 2-inserter:
 consider the functors $F,G : \TT \to \Cat$ given by $F(C,\Ty,\Tm) \eqdef C$ and $G(\_) \eqdef \op{\Set}$. As pseudotransformations, we take the projections $\alpha \eqdef \Tm$ and $\beta \eqdef \Ty$, respectively.
 As in \Cref{ex:1-inserters}, we have to put the oppositization into the target pseudofunctor $G$, that is, take presheaves on $C$ to be functors $C \to \op{\Set}$ instead of $\op{\C} \to \Set$.
\end{example}

\begin{proposition}[\coqident{Bicategories.DisplayedBicats.Examples.Displayed2Inserter}{disp_two_inserter_univalent_2_0}]
 In the context of \Cref{def:disp-2-inserter}, the displayed 2-inserter is
 \begin{itemize}
  \item locally univalent;
  \item globally univalent if $\B$ is locally univalent.
 \end{itemize}

\end{proposition}

\section{Conclusions and Open Questions}
\label{sec:conclusion}

In the present work, we studied univalent bicategories. 
Showing that a bicategory is univalent can be challenging; to simplify this task, we 
introduced displayed bicategories, which provide a way to modularly reason about 
involved bicategorical constructions.
We then demonstrated the usefulness of displayed bicategories by using them to show
that certain complicated bicategories are univalent.
The same approach is useful for many other basic notions and constructions such as
pseudofunctors, pseudotransformations, modifications, and biequivalences:
the displayed machinery allows one to stratify their presentation and thus eases reasoning on such objects.
Veltri and Van der Weide~\cite{VvdW} used the techniques described in the present paper to construct univalent bicategories of algebras for a class of signatures.
In addition, they defined displayed biadjunctions, and those were used to construct biadjunctions between bicategories of algebras.

For the practical mechanization of mathematics in a computer proof assistant,
two issues may arise when building an elaborate bicategory as the total bicategory of iterated displayed bicategories.
Firstly, the structures may not be parenthesized as desired. This problem can be avoided or at least alleviated through
a suitable use of the sigma construction of displayed bicategories (\Cref{ex:disp_sigma} in \Cref{ex:disp_bicat}).
Secondly, ``meaningless'' terms of unit type may occur in the cells of this bicategory. We are not aware of a way of avoiding these occurrences while still using displayed bicategories.
However, both issues can be addressed through the definition of a suitable ``interface'' to the structures, in form of 
``builder'' and projection functions, which build, or project a component out of, an instance of the structure.
The interface hides the implementation details of the structure, and thus provides a welcome separation of concerns between 
mathematical and foundational aspects.

We have only started, in the present work, the development of bicategory theory in 
univalent foundations and its formalization.
There are some important questions that we have left open, such as proving the universal property of the Rezk completion.
Furthermore, the precise relationship to the bicategories studied in \cite[Example~9.1]{univalence-principle} should be established; those bicategories are defined, in particular, using relations instead of functions.
It seems reasonable to hope for our univalent bicategories to coincide (in the sense of an equivalence of types) with the univalent bicategories of \cite[Example~9.1]{univalence-principle}; a construction of such an equivalence is outside the scope of this work.
We also anticipate that the displayed machinery can be usefully employed for extending 
the comparison of different categorical structures for type theories started by Ahrens, Lumsdaine, and Voevodsky~\cite{voevodsky2018categorical} to the bicategorical setting.




\bibliography{literature}

\end{document}